\documentclass[11pt,  a4paper]{amsart}
\usepackage{eucal}
\usepackage{graphicx}
\usepackage[english]{babel}
\usepackage{amsmath,amsthm,amssymb,amsfonts,amscd,amsbsy,amsxtra}
\usepackage{epsfig,color}
\usepackage{subfigure}
\usepackage{psfrag}
\usepackage[arrow,matrix, curve]{xy}
\usepackage{amsmath}
\usepackage{enumerate}
\usepackage[T1]{fontenc} 
\usepackage[utf8]{inputenc}
\usepackage{mathrsfs}
\usepackage[shortlabels]{enumitem}
\usepackage{mathtools}
\usepackage[bookmarks=true,pdfstartview=FitH, pdfborder={0 0 0}, colorlinks=true,citecolor=red, linkcolor=blue]{hyperref}
\usepackage{aliascnt}
\usepackage{a4wide}
\usepackage{nicefrac}

\newcommand {\R}	{\mathbb{R}}
\newcommand {\N}	{\mathbb{N}}

\newcommand {\C}	{\mathbb{C}}

\DeclareMathOperator{\IBC}{IBC}
\newcommand{\HIBC}{H_{\IBC}^{0,1}}
\newcommand{\HIBCab}{H_{\IBC}^{\alpha,\beta}}

\newcommand{\HIBCb}{H_{\IBC}^{0,\beta}}
\newcommand{\Glambda}{G_\lambda}
\newcommand{\Glambdaab}{G_\lambda^{\alpha,\beta}}

\DeclareMathOperator{\rg}{rg}

\DeclareMathOperator{\Id}{Id}

\renewcommand{\H}{\mathscr{H}}
\newcommand{\dH}{\partial \H}

\renewcommand{\epsilon}{\varepsilon}

\setlist[enumerate]{font = \normalfont}

\setlist[enumerate]{font = \normalfont}

\theoremstyle{plain}
\newtheorem{thm}{Theorem}[section]
\newaliascnt{cor}{thm}
\newaliascnt{prop}{thm}
\newaliascnt{lem}{thm}
\newtheorem{cor}[cor]{Corollary}
\newtheorem{prop}[prop]{Proposition}
\newtheorem{lem}[lem]{Lemma}

\aliascntresetthe{cor}
\aliascntresetthe{prop}
\aliascntresetthe{lem} 
\newcounter{stp}
\newcounter{stpi}
\newcounter{stpci}
\newcounter{stpiii}

 \theoremstyle{theorem}
\newtheorem{step}[stp]{Step}

\theoremstyle{definition}
\newaliascnt{defn}{thm}
\newaliascnt{asu}{thm}
\newaliascnt{con}{thm}
\newaliascnt{sett}{thm}
\newaliascnt{constr}{thm}
\newtheorem{defn}[defn]{Definition}
\newtheorem{asu}[asu]{Assumption}

\newtheorem{sett}[sett]{Abstract Setting}
\newtheorem{constr}[constr]{Construction}
\aliascntresetthe{defn}
\aliascntresetthe{asu}
\aliascntresetthe{con}
\aliascntresetthe{sett}
\aliascntresetthe{constr} 
\theoremstyle{remark}
\newaliascnt{rem}{thm}
\newaliascnt{exa}{thm}
\newaliascnt{masu}{thm}
\newaliascnt{nota}{thm}
\newtheorem{rem}[rem]{Remark}
\newtheorem{exa}[exa]{Example}

\newtheorem{nota}[nota]{Notation}
\aliascntresetthe{rem}
\aliascntresetthe{exa}
\aliascntresetthe{masu}
\aliascntresetthe{nota}
\numberwithin{equation}{section}

\date{\today}%

\begin{document}

\title{An abstract framework for interior-boundary conditions}

\author{Tim Binz}
\address{Tim Binz,  Princeton University, Program in Applied and Computational Mathematics,
Fine Hall, Washington Road, 08544 Princeton NJ, USA}
\email{tb7523@princeton.edu}

\author{Jonas Lampart}
\address{Jonas Lampart, CNRS \& LICB, UMR 6303 Université Bourgogne Europe, 9 Av. A. Savary, 21078 Dijon, France}
\email{lampart@math.cnrs.fr}

\date{\today}%

\begin{abstract}
In a configuration space whose boundary can be identified with a subset of its interior, a boundary condition can 
relate the boundary values of a function to its values in the interior. Additionally, boundary values can 
appear as additive perturbations. 
Such boundary conditions have recently provided insight into problems from quantum field theory.
We discuss interior-boundary conditions in an abstract setting, with a focus on self-adjoint operators,  proving self-adjointness criteria, resolvent formulas, and a classification theorem.
\end{abstract}

\maketitle

\section{Introduction}

Consider a differential operator on a configuration space consisting of a disjoint union of manifolds with boundary (or 
corners) of different dimensions,
\begin{equation}
 M=\bigsqcup\limits_{n=1}^N M_n.
\end{equation}
If there is a map $\iota_n:\partial M_n \to M_{n-1}$, a boundary condition may relate boundary values of  a function on 
$M_n$ to to the values on $M_{n-1}\supset \iota_n (\partial M_n)$. We call such a boundary 
condition an interior-boundary condition following Teufel and Tumulka~\cite{TeTu21, TeTu16, tumulka2020}. If, for 
example, $\iota_n$ is bijective, we may also add boundary value operators, such as $f\vert_{\partial M_n}\circ\iota_n 
^{-1}$ as perturbations to the differential operator. This gives rise to a coupled system of equations for functions 
$f_j=f\vert_{M_j}$.

A simple example of such a setup is obtained by taking $M_n=(\R_+)^n$ for $n= 0,1$, i.e., the half-line $M_1=\R_+$, with boundary $M_0=\{0\}=\partial M_1$. As a differential operator $L$ on $M=M_0\sqcup M_1$ we may take the (negative) Laplacian on
$\R_+$, extended to $M$ by setting $Lf\vert_{M_0}=0$. Taking a self-adjoint boundary condition for the Laplacian gives rise to decoupled equations for $f_0, f_1$.
However, one can couple the two functions by a boundary condition such as
$f'_1(0)=f_0$, which gives rise to a symmetric operator if one adds boundary values to the differential expression in the following way.
Consider the operator
\begin{equation}
 H(f_0,f_1)= (f_1(0), -f''_1) = L f + If_1(0),
\end{equation}
subject to this boundary condition ($I$ denotes the inclusion of the first summand). Then $H$ is symmetric with respect to the canonical scalar product on $L^2(M)\cong
\C\oplus L^2(\R_+)$, since
\begin{align}
 \langle f, Hf \rangle &= \overline f_0 f_1(0) - \int_0^\infty \overline{f_1}(x) f_1''(x) dx \\
&= \int_0^\infty|f_1'(x)|^2 dx + \overline f_0 f_1(0) + \overline f_1(0) f_1'(0)\notag\\
& =  \int_0^\infty|f_1'(x)|^2 dx + 2\mathrm{Re}(\overline f_0 f_1(0)) \in \R,\notag
\end{align}
and one easily checks that $H$ is self-adjoint on a suitable domain~\cite{yafaev1992}.

We will work in a more abstract framework, where the configuration space actually plays no role. We rather consider directly two Hilbert spaces $\H$ 
and $\dH$, where the first would correspond to $L^2(M)$ and the second to $L^2(\partial M)$.
Consider a densely defined ``maximal'' operator $L_m$ on $\H$ and two ``boundary value operators'' $A_m, B$ mapping (a 
subset of) $D(L_m)$ to $\dH$. This is a standard setup for abstract boundary value problems, see \cite{Gre:87, ABE:13, ABE:16, AE:18},
and dynamical boundary conditions, see \cite{CENN:03, EF:05,BE:19}.
Further, on Hilbert spaces, such abstract boundary problems are related to the theory of quasi boundary triples, see \cite{BeLa2007, BeGeNa17, BHS:20} and \hyperref[appendix]{Appendix \ref*{appendix}}.

In addition to the usual ingredients, we assume that we are given a bounded operator 
\begin{equation*}
 I:\dH \to \H.
\end{equation*}
This operator is the characteristic 
feature of the interior-boundary conditions, since it allows for the formulation of conditions relating elements of 
$\H$ and $\dH$.  
Most of the theory we develop reduces to the usual theory of boundary conditions with the choice $I=0$.
Non-trivial examples where such a structure is relevant are hierarchies of boundary value problems. In this case, $\H$ is
a finite or countable direct sum of spaces $\H_n$, $L_m$ is an operator on $\H_n$ for each $n$, $A_m,B$ map $\H_n$ to 
$\partial \H_n$, and $I:\partial \H_n \to \H_{n-1}$ is an isomorphism. On this space we can consider operators such as 
$H=L_m + IA_m$ subject to boundary conditions, such as $Bf=I^*f$. Spelling out the equation $Hf=g$ on $\H_n$, it reads
\begin{equation}
 L_m f_n + I A_m f_{n+1} = g_n,
\end{equation}
with the boundary condition
\begin{equation}
 Bf_n= I^*f_{n-1}.
\end{equation}
The unknown $f_n$ is thus coupled to $f_{n-1}$ by the boundary condition and to $f_{n+1}$ by the operator $IA:\H_{n+1} \to \H_n$.

Formulations of different models from quantum field theory (QFT) in terms of such hierarchies have been 
proposed by Landau and Peierls~\cite{landau1930}, Moshinsky and Lopez~\cite{moshinsky1991}, Yafaev~\cite{yafaev1992}, 
Teufel and Tumulka~\cite{TeTu21, TeTu16}. There, $\H_n$ is the kinematical  Hilbert space for $n$ (indistinguishable) 
particles and possibly some additional particles of a different type, e.g., $n$ photons and a fixed number of electrons.
The spaces with different numbers of particles are coupled since, in quantum field theory, particle numbers are not 
conserved. In this context $A_m$ is closely related to the so-called annihilation operator (or a power thereof), an unbounded operator that 
reduces particle-number by one, and the boundary condition incorporates the process of particle creation.
Similar (finite) hierarchies have also been studied as models for nuclear reactions~\cite{moshinsky1951, 
moshinsky1951b, moshinsky1951c, thomas1984}.

In quantum mechanics, the dynamics of a system are generated by a self-adjoint operator. However, constructing these 
and proving self-adjointness for quantum field theoretic models poses many difficulties. 
One of these is the problem of ultra-violet singularities that stem from the distributional nature of the interactions.
Interior-boundary conditions have proved to be an effective way of addressing these singularities. They provide an 
alternative to renormalisation techniques going back to Nelson~\cite{nelson1964} and Eckmann~\cite{eckmann1970}, with 
the benefit of giving a direct description of the domain of self-adjointness~\cite{KeSi16, IBCpaper, LaSch19, 
IBCmassless, IBCrelat}.
These ideas were later extended to more singular models by the second author~\cite{La19, La20, La23}.
Other models with interior-boundary conditions were investigated in~\cite{tumulka2005, lienert2019, schmidt2019, HLM25}.

These constructions of self-adjoint operators are related to singular number-preserving interactions that 
can be described by (generalised) boundary conditions and classified in terms of self-adjoint extensions of certain 
``minimal'' operators (see e.g.~\cite{albev, Co_etal15, Be_etal2017, posi08}, and references therein). Such methods were applied to an abstract form of interior-boundary conditions by Posilicano~\cite{posi20} (see~\autoref{rem:posi} for comparison to our approach).

Our goal in this article is to develop a general theory of interior-boundary conditions.
In \hyperref[sect:frame]{Section \ref*{sect:frame}} we explain the abstract framework. In \hyperref[sect:Robin]{Section \ref*{sect:Robin}} we discuss operators with ``Robin type'' boundary conditions of the form $\alpha A_mf+ \beta Bf=I^*f$, their symmetry and self-adjointness.
In \hyperref[sect:ext]{Section \ref*{sect:ext}} we consider more general boundary conditions. In particular, we construct a quasi boundary triple (see~\hyperref[appendix]{Appendix \ref*{appendix}}) that allows us to relate operators with different boundary condtions and classify certain self-adjoint conditions. 
In \hyperref[sect:applications]{Section \ref*{sect:applications}} we give a non-trivial example to which our theory applies.

\section{Abstract Framework}\label{sect:frame}

In this section we introduce an abstract framework to formulate interior-boundary conditions and some notational conventions.
 
\begin{nota}
Let $X$ and $Y$ Banach spaces and $T \colon D(T) \subset X \to  Y$ a densely defined operator. 
For $\lambda \in \rho(T)$, we denote the resolvent of $T$ by 
\begin{equation}
 R(\lambda, T):=(\lambda - T)^{-1} \in \mathcal{L}(Y,X).
\end{equation}
For $\lambda \in \C\setminus \left(\sigma_p (T)\cup\sigma_r(T) \right)$, the algebraic inverse of $\lambda-T$ is a densely defined operator with $D((\lambda-T)^{-1})=\rg(\lambda-T)$, for which we use the notation $(\lambda-T)^{-1}$.

If $X$ and $Y$ are Hilbert spaces, we denote the adjoint of $T$ by $T^* :D(T^*) \subseteq Y \to X$. 

Moreover let $Z$ another Banach space and $S \colon D(S) \subset Y \to Z$ a densely defined operator. The composition $S T \colon D(ST) \subset X \to Z$ is the (not necessarily densely defined) operator given by
\begin{align}
	(ST) x := S(Tx), \qquad D(ST) := \{ x \in D(T) \colon Tx \in D(S) \} .
\end{align}
 \end{nota}

\begin{sett}\label{sett:abstract setting}
As a starting point of our investigation, we assume that we are given the following objects:
	\begin{enumerate}[(i)]
		\item two Hilbert spaces $\H$ and $\dH$;
		\item a ``maximal'' operator $L_m \colon D(L_m) \subset \H \to \H$;
		\item a trace operator $B \colon D(L_m) \subset \H \to \dH$;
		\item a boundary operator $A_m \colon D(A_m) \subset \H \to \dH$;
		\item a bounded ``identification''  operator $I \colon \dH \to \H$.
	\end{enumerate}
\end{sett}

We then denote by $L$ the restriction of $L_m$ to the kernel of $B$
\begin{equation}
 L:=(L_m)|_{\ker(B)}
\end{equation}
and also
\begin{equation}
 A := (A_m)|_{\ker(B)}.
\end{equation}


We assume these to have the following properties.

\begin{asu}\label{ass:general}
	\begin{enumerate}[(a)]
		The operators $L, A, B$ satisfy:
		\item $L$ is self-adjoint
		\item 
		the operator $A$ is relatively $L$-bounded;
		\item 
		for $\lambda \in \rho(L)$: $\rg\Big((A R(\lambda, L))^*\Big) \subseteq \ker(\overline{\lambda} - L_m)$
		;
		\item for $\lambda \in \rho(L)$: $B (A R(\lambda, L))^* = \Id_{\dH} \label{eq:B,L,A}$. 
	\end{enumerate}
\end{asu}

\begin{rem}\label{rem:1}
	Note that \autoref{ass:general}~(d) implies that $(A R(\lambda, L))^*$ is injective and that $B$ is surjective. 
	Since, in general, $\rg(T)^\perp= \ker(T^*)$, $\rg(A) \subset \dH$ is dense. 
\end{rem}

\begin{defn}
	For $\overline{\lambda} \in \rho(L)$ we define the \emph{abstract Dirichlet operator} associated with $\lambda$ as 
	\begin{equation}
	\Glambda := (A R(\overline{\lambda}, L))^* .
	\end{equation} 
	Moreover we define the \emph{abstract Dirichlet-to-Neumann operator} associated with $\lambda$ by
	\begin{equation}
	T_\lambda := A_m \Glambda , \quad
	D(T_\lambda) = \{ \varphi \in \dH\colon \Glambda \varphi \in D(A_m) \} .
	\end{equation}
\end{defn}

\begin{exa}
 An example where our assumptions are satisfied is obtained from the Laplacian on the half line. Let $\H=\C \oplus L^2(\R_+)$, $\partial\H=\C$ and $I:\C\to \H$ be inclusion of the first summand. If we let $L_m(f,z)=(0,-\partial_x^2f)$ with $D(L_m)=\C\oplus H^2(\R_+)$, and choose $B(f,z)=f(0)$ to be the Dirichlet boundary value and $A(f,z)=-f'(0)$ as the negative Neumann value (the roles of these two may also be exchanged). We have $D(L)=\C\oplus \big(H^2(\R_+)\cap H^1_0(\R_+)\big)$ so $L$ is essentially the Dirichlet Laplacian, which is self-adjoint. The Neumann value is controlled by $L$ as a consequence of the Sobolev embedding, and $G_\lambda z= (0,z  e^{-\sqrt{-\lambda} x})\in \ker (\lambda-L_m)$ (for $\lambda\in \C\setminus [0, \infty)$, where $\mathrm{Re}\sqrt{-\lambda}>0$) and $BG_\lambda z=z$, so Assumption~\ref{ass:general} is satisfied.
\end{exa}

\begin{rem}
	Note that by \autoref{ass:general}~(b) the abstract Dirichlet operator is bounded, $\Glambda \in \mathcal{L}(\dH,\H)$. Further by \autoref{ass:general}~(c) it satisfies $\rg(\Glambda) \subseteq \ker(\lambda-L_m)$ and by \autoref{ass:general}~(d) it is the right-inverse of $B$. 
	Our definition thus coincides with the common definition of a Dirichlet operator in the literature, e.g.~\cite{Gre:87}.
\end{rem}

The next Proposition collects some direct consequences of our general assumptions that will play an important role throughout.

\begin{prop}\label{prop:rechenregeln}
Under~\autoref{ass:general} we have:
	\begin{enumerate}[(i)]
		\item 	
		The domain $D(L_m)$ of the maximal operator can be decomposed into
		\begin{equation}
		D(L_m) = D(L) \oplus \ker(\lambda - L_m) . \label{eq:decomposition}
		\end{equation} 
		The projections are given by $\Glambda B \colon D(L_m) \to \ker(\lambda-L_m)$ and $(\Id_{D(L_m)}-\Glambda B)~\colon~D(L_m)~\to~D(L)$. 
		\item 
		The following identity holds
		\begin{equation*}
		G_\lambda^* (\overline{\lambda} - L) f = A f 
		\end{equation*}
		for $f \in D(L)$. 
		\item For $\lambda, \mu \in \rho(L)$ the domains of the Dirichlet-to-Neumann operators coincide, i.~e.~$D(T_\lambda) = D(T_\mu)$ (we will thus simply denote this domain by $D(T)$). Moreover their difference, given by
		\begin{equation*}
		T_\lambda - T_\mu = (\mu- \lambda) A R(\mu, L) \Glambda,
		\end{equation*}
		is bounded.
	\end{enumerate}
\end{prop}
\begin{proof}
	\begin{enumerate}[(i)]
		\item 
		Since $B G_\lambda = \Id_{\dH}$, $G_\lambda B$ and $(\Id_{D(L_m)} - G_\lambda B)$ are projections on $D(L_m)$, and (algebraically) $D(L_m) = \rg(G_\lambda B)  \oplus \rg(\Id-G_\lambda B)$. 
		\autoref{ass:general}~(b) means that the image satisfies $\rg(\Glambda B) \subseteq \ker(\lambda- L_m)$. 
 		Further, we clearly have 
 		\begin{equation}
 		\rg(\Id_{D(L_m)}-\Glambda B) \subseteq \ker(B)=D(L).
 		\end{equation}
		Therefore, using
		$\ker(\Glambda B) = \rg(\Id_{D(L_m)}-\Glambda B) \subseteq D(L)$, we obtain
			\begin{equation}
				D(L_m) = \rg(\Id_{D(L_m)}-\Glambda B) \oplus \rg(\Glambda B) 
				\subseteq D(L) + \ker(\lambda - L_m) \subseteq D(L_m).
			\end{equation}
		Since $\lambda \in \rho(L)$ by assumption, the latter sum is direct and we have  
			\begin{equation}
				D(L_m) = D(L) \oplus \ker(\lambda - L_m) . 
			\end{equation}
		\item The definition of $\Glambda$ implies for $f \in D(L), g \in \H$:
		\begin{align}
		\langle (\overline{\lambda} - L) f, G_\lambda g \rangle_{\H}
		&= \langle (\overline{\lambda} - L) f, (AR(\overline{\lambda},L))^* g \rangle_{\H} \notag\\
		&= \langle (AR(\overline{\lambda},L)) (\overline{\lambda} - L) f, g \rangle_{\H} \notag\\
		&= \langle A f , g \rangle_{\H}.
		\end{align}
		\item From resolvent resolvent identity it follows
		\begin{equation*}
		G_\lambda^* - G_\mu^*
		= A(R(\overline{\lambda},L)-R(\overline{\mu},L)) 
		=  (\overline{\mu}-\overline{\lambda}) A R(\overline{\lambda},L) R(\overline{\mu},L) .  
		\end{equation*}
		Using the self-adjointness of $L$ we conclude that
		\begin{equation}\label{eq:G-diff}
		G_\lambda - G_\mu 
		=  (\mu-\lambda) R(\overline{\mu},L)^* (A R(\overline{\lambda},L))^*
		= (\mu-\lambda)) R(\mu,L) G_\lambda . 
		\end{equation}
		Since $A$ is relatively $L$-bounded, the term on the right hand side satisfies
		\begin{equation*}
		\rg( (\mu-\lambda) R(\mu,L) G_\lambda ) \subseteq D(L) \subset D(A)
		\end{equation*}
		and the first claim follows. The identity for the difference follows from the definition of $T_\lambda$, $T_\mu$ and the boundedness from the fact that $A$ is $L$-bounded. 
	\end{enumerate}
\end{proof}

We now give a construction procedure that leads to operators $L_m$, $B$, $A_m$ etc. with the properties of \autoref{sett:abstract setting}.
This construction can be applied in many concrete cases. A non-trivial example is given below.
The logic here is somewhat different than in the definitions, in that we start with  the operators $L, A, I$ and $T_\lambda$ (for one, arbitrarily fixed $\lambda \in \rho(L)$). From these, $L_m$, $B$ and $A_m$ are constructed as follows.

\begin{constr}\label{construction}
	We are given two Hilbert spaces $\H$, $\dH$ and a bounded operator $I \colon \dH\to \H$. 
	Further, we have a self-adjoint operator $L \colon D(L) \subset \H \to \H$, a relatively $L$-bounded operator $A \colon D(A) \to \partial\H$ and a closed operator	$T \colon D(T) \subset \dH\to \dH$. 
	To construct the operators $L_m$, $B$ and $A_m$ we proceed by the following steps:
	\begin{step}
		Consider the ``minimal''  operator $L_0 \colon D(L_0) \subset \H\to \H$, defined by
		\begin{equation}
		L_0 f = L f, \qquad D(L_0) = D(L) \cap \ker(A) = \ker(A) .
		\end{equation}
	\end{step}
	\begin{step}\label{stp:2 constr}
		Assume that $\ker(A)$ is dense, so the adjoint $L_0^\ast$ is well defined.
		Let $\lambda \in \rho(L)$, and, since $L$ is self-adjoint, also $\overline{\lambda} \in \rho(L)$. 
		By \autoref{prop:rechenregeln}~(ii), the operator $\Glambda$ given by $\Glambda := (A R(\overline{\lambda},L))^*$
		satisfies
		\begin{equation}
		\langle (\lambda - L_0^\ast) \Glambda f , \varphi \rangle_{\H}
		= \langle f , \Glambda^*(\overline{\lambda}-L_0) \varphi \rangle_{\H}
		= \langle f , A \varphi \rangle_{\dH} = 0 
		\end{equation}
		for $f \in \H$ and $\varphi \in D(L_0) \subset \ker(A)$. So $\rg(\Glambda) \subseteq \ker({\lambda}-L_0^\ast)$ and, since $\lambda \in \rho(L)$, the operator 
		\begin{equation}
		L_m f := L_0^\ast f, \qquad D(L_m) := D(L) \oplus \rg(\Glambda)
		\label{eq:decomposition-constr} .
		\end{equation} 
		is well defined and \autoref{ass:general}~(c) is satisfied. Note that, since $\rg(G_\mu-G_\lambda)\subset D(L)$ by~\eqref{eq:G-diff}, the right hand side is independent of $\lambda$. 
	\end{step}
	\begin{step}
		Now we define $B \colon D(L_m) \to \dH$ as the left-inverse of $G_\lambda$, i.e.~using the unique decomposition $f=f_0+ \Glambda \varphi$, $f_0\in D(L)$, $\varphi \in \rg(G_\lambda)$, we set $Bf = B (f_0 + \Glambda \varphi) := \varphi$, which satisfies \autoref{ass:general}~(d), and $L=(L_m)|_{\ker(B)}$.
	\end{step}	
	\begin{step}
		Let $\lambda \in \rho(L)$ and $T \colon D(T) \subset \dH\to \dH$ a fixed operator. 
		We define the operator $A_m \colon D(A_m) \subset \H \to \dH$ by
		\begin{equation}
		A_m f := A f_0 + T \varphi ,\qquad D(A_m) = D(L) \oplus \Glambda D(T) \subset D(L_m) .
		\end{equation}
		Hence $A := (A_m)|_{\ker(B)}$, so $A_m$ extends $A$, and $T_\lambda = A_m \Glambda = T$. 
		The operators $T_\mu$ for $\mu\ne \lambda$ are then dertemined by~\autoref{prop:rechenregeln}~(iii).
	\end{step} 
\end{constr}

\setcounter{stp}{0}

The following example is essentially the model considered by Moshinsky~\cite{moshinsky1951, moshinsky1951b, moshinsky1951c} and Yafaev~\cite{yafaev1992}.

\begin{exa}[Moshinsky-Yafaev model]\label{exa:yaf}
	Set $\H=\C\oplus L^2(\R^3)$, $\dH=\C$, $I z = (z, 0)$.
	Define on $D(L)=\C\oplus H^2(\R^3)$
	\begin{equation}
	L(f,z)= (0,-\Delta f).
	\end{equation}
	Let $A : D(L)\to \dH$ be given by $A(f,z)=f(0)$.
	
	\begin{step}
		The operator $L_0$ is the restriction of $L$ to 
		$D(L_0)=\C\oplus H^2_0(\R^3\setminus\{0\})$.
		Note that this operator is densely defined. 
	\end{step}
	\begin{step}
		The domain of the adjoint is given by 
		\begin{equation}
		D(L_0^*)=D(L)\oplus \mathrm{span}(0,g_\lambda),
		\end{equation}
		with, for any $\lambda \in \rho(L)=\C\setminus\R_+$ (taking the branch of the square root with positive real part)
		\begin{equation}
		g_\lambda(x)=-\frac{e^{-\sqrt{-\lambda} |x|}}{4\pi |x|}.
		\end{equation}
		Moreover, we have $G_\lambda z = (0,z g_\lambda)$.
		Hence we set\footnote{Note that, in general, the operator $L_0^\ast$ is \glqq too big\grqq, in the sense that not all functions in $D(L_0^\ast)$ have boundary values in $\dH$, e.g. if $\dH=L^2(\partial M)$.}
		$L_m := L_0^\ast$,
		which acts as $L_0^*(f,z)= \big(0,-\Delta_0^*f\big)$, where $-\Delta_0^*$ is the adjoint of $-\Delta\vert_{H^2_0(\R^3\setminus \{0\})}$\end{step}
	\begin{step}
		The operator $B$, defined as the left-inverse of $G_\lambda$, is given by the formula
		\begin{equation}
		B(f,z)= -4\pi \lim_{x\to 0}|x|f(x). 
		\end{equation}
	\end{step}
	\begin{step}
		Since $g_\lambda$ is not continuous in $x=0$ we cannot define $A_m$ as the evaluation at $x=0$. 
		However, the following formula, which extends the evaluation, is well defined on $D(L_m)$
		\begin{equation}
		A_m(f,z)= A_mf :=\lim_{r\to 0} \partial_r  r\frac{1}{4\pi } \int_{S^2}f(r \omega) d\omega.
		\end{equation}
		This yields the formula for $T_\lambda:\C\to \C$
		\begin{equation}
		T_\lambda = A_m G_\lambda =\lim_{r\to 0}\partial_r  \left(-\frac{e^{-\sqrt {-\lambda} r}}{4\pi} \right)  
		=\frac{\sqrt{-\lambda}}{4\pi}.
		\end{equation}
	\end{step}	
	
With this framework in place, the operators with interior-boundary conditions take the form 
\begin{align}
 \HIBCab (f , z )
 &=( -\Delta_0^* f ,  \gamma A_m f + \delta B f )\\
 D(\HIBCab)&=\{ (f,z)\in D(\Delta_0^*)\oplus \C: \alpha A_m f + \beta Bf =z \},
\end{align}
with complex numbers $\alpha, \beta, \gamma, \delta \in \C$.
One easily checks that these operators are symmetric iff $\bar{\alpha} \gamma, \bar{\beta} \delta \in \R$ and $\beta \bar{\gamma} - \bar{\alpha} \delta = 1$ (see also~\autoref{lem:symmetry}). It is also not difficult to show that these symmetric operators are self-adjoint, see~\cite{yafaev1992}.
Note that if instead of our choice of $I$ we would have taken $I=0$, we would have found the ``Laplacian with $\delta$-potential''~\cite{albev}, which is a well known example in the theory of singular boundary value problems.
We will use this example throughout the article to illustrate our results.
\end{exa}

\section{Interior-Boundary Conditions of Robin Type}\label{sect:Robin}

In this section we will discuss a simple family of interior-boundary conditions in which the boundary operators $A_m$ and $B$ are related to the values in the interior simply by by some constants, exactly as in~\autoref{exa:yaf} (more general conditions are considered later, in Section~\ref{sect:ext}). We then investigate symmetry and self-adjointness of these operators and prove various formulas for their resolvents.

Here, as always, we work within the framework introduced in~\autoref{sett:abstract setting} and~\autoref{ass:general}.

\begin{defn}
	Let $\alpha, \beta, \gamma, \delta \in \C$. The operators with \emph{ interior-boundary conditions} (abbreviated IBCs) 
	of type $(\alpha, \beta)$, denoted $\HIBCab \colon D(\HIBCab) \subset \H \to \H$  are defined by
		\begin{align}\label{eq:HIBCab}
			\HIBCab f &:= L_m f + \gamma I A_m f + \delta I B f, \\
			D(\HIBCab) &:= \{ f \in D(L_m) \cap D(A_m) \colon \alpha A_m f + \beta B f = I^* f \} . \notag  
		\end{align}	
\end{defn}

Note that $\HIBCab$ is really a family of operators depending on $\gamma, \delta$. However, since the values of $\gamma, \delta$ play only a minor role we suppress them in the notation. 
Note also that, up to a bounded perturbation, we can always assume that $\delta=0$ (if $\beta\neq 0)$ or $\gamma=0$ (for $\alpha\neq 0$), since by the boundary condition (e.g. for $\beta\neq 0$)
\begin{equation*}
 \delta I B \vert_{D(\HIBCab)} = \delta \beta^{-1} \left(   I I^* - \alpha I A_m\right)\vert_{D(\HIBCab)}.
\end{equation*}

\subsection{Symmetry}

We start by investigating the elementary properties of $\HIBCab$, in particular symmetry. For this we will make the following additional assumption for the remainder of the article.

\begin{asu}
	For all $\lambda \in \rho(L)$, we have $T_{\overline{\lambda}} \subset T_\lambda^\ast$. In particular, $T_\lambda$ is symmetric on $\dH$ for $\lambda \in \R\cap \rho(L)$. 
\end{asu}

Note that, since $L$ is self-adjoint by~\autoref{ass:general}~a), $\lambda \in \rho(L)$ implies $\overline{\lambda} \in \rho(L)$, so the assumption makes sense.
By~\autoref{prop:rechenregeln}~(iii), if $T_{\overline{\lambda}} \subset T_\lambda^\ast$ for one $\lambda\in \rho(L)$, then this automatically holds for all $\lambda \in \rho(L)$.

With this assumption, it is easy to show an \emph{abstract Green identity}, which essentially generalises integration-by-parts for Laplace-type operators to our abstract setting. 

\begin{lem}\label{lem:integration by parts}
The following identity holds for all $f, g \in D(L_m) \cap D(A_m)$
	\begin{align*}
	\langle L_m f , g \rangle_\H - \langle f, L_m g \rangle_{\H} 
	= \langle B f, A_m g \rangle_{\dH} - \langle A_m f , Bg \rangle_{\dH} . \label{eq:integration by parts}
	\end{align*}
\end{lem}
\begin{proof}
	Let $\lambda \in \rho(L)$ and $f, g \in D(L_m) \cap D(A_m)$. Note that $\overline{\lambda} \in \rho(L)$. 
	Using $D(L_m) = D(L) \oplus \ker(\lambda-L_m)$, $\rg(\Glambda) \subseteq \ker(\lambda-L_m)$ and \autoref{prop:rechenregeln}~(ii) we obtain
	\begin{align}
	\langle (L_m -\lambda) f , g \rangle_{\H} \notag
	&= \langle (L_m-\lambda) (\Id - \Glambda B)f, (\Id - G_{\overline{\lambda}} B)g + G_{\overline{\lambda}} B g \rangle_{\H} \\
	&= \langle (L-\lambda) (\Id - \Glambda B)f, (\Id - G_{\overline{\lambda}} B)g \rangle_{\H}
	+ \langle (L-\lambda) (\Id - \Glambda B)f, G_{\overline{\lambda}} Bg \rangle_{\H} \notag \\
	&= \langle (L-\lambda) (\Id - \Glambda B)f, (\Id - G_{\overline{\lambda}} B)g \rangle_{\H}
	+ \langle (G_{\overline{\lambda}})^* (L-\lambda) (\Id - \Glambda B)f, Bg \rangle_{\dH} \notag \\
	&= \langle (L-\lambda) (\Id - \Glambda B)f, (\Id - G_{\overline{\lambda}} B)g \rangle_{\H}
	- \langle A (\Id - \Glambda B)f, Bg \rangle_{\dH} \notag \\		 	
	&= \langle (L-\lambda) (\Id - \Glambda B)f, (\Id - G_{\overline{\lambda}} B)g \rangle_{\H}
	- \langle A_m f, Bg \rangle_{\dH}
	+ \langle T_\lambda Bf, Bg \rangle_{\dH} .
	\end{align}
	By an analogous calculation we obtain
	\begin{align}
	\langle f , (L_m- \overline{\lambda}) g \rangle_{\H}
	= \langle (\Id - \Glambda B)f, (L-\overline{\lambda})(\Id - G_{\overline{\lambda}} B)g \rangle_{\H}
	- \langle B f, A_m g \rangle_{\dH}
	+ \langle Bf, T_{\overline{\lambda}} Bg \rangle_{\dH} .
	\end{align}
	By the symmetry of $L$ and $T_{\overline{\lambda}} \subset T_\lambda^\ast$, taking the difference of these two equations proves the claim. 
\end{proof}

With this result we can easily determine when $\HIBCab$ is symmetric. Conditions of this type were also given in~\cite[eq.~(8)-(10)]{tumulka2020}. 
The necessity of these conditions will be further addressed in the more general framework of~\hyperref[sect:ext]{Section \ref*{sect:ext}}.

\begin{lem}\label{lem:symmetry}
	The operators $\HIBCab$ are symmetric on $\H$ if
	\begin{equation*}
	\bar{\alpha} \gamma, \bar{\beta} \delta \in \R \qquad \text{ and } \qquad  \beta \bar{\gamma} - \bar{\alpha} \delta = 1. 
	\end{equation*} 
\end{lem}
\begin{proof}
	From \autoref{lem:integration by parts} we conclude
	\begin{align}\label{eq:1 lem:symmetry}
		\langle \HIBCab f, g \rangle_{\H} - \langle f , \HIBCab g \rangle_{\H}
		=& \langle B f, A_m g \rangle_{\dH} 
		- \langle A_m f , B g \rangle_{\dH} \\
		&+ \gamma \langle I A_m f , g \rangle_{\H} 
		+ \delta \langle I B f , g \rangle_{\H} \notag \\
		&- \bar{\gamma} \langle f , I A_m g \rangle_{\H} 
		- \bar{\delta} \langle f , I B g \rangle_{\H}  \notag 
	\end{align}
	for $f, g \in D(\HIBCab)$. The IBC $\alpha A_m f+\beta Bf= I^*f$ now implies
	\begin{align}\label{eq:2 lem:symmetry}
		\gamma& \langle I A_m f , g \rangle_{\H} 
	+ \delta \langle I B f , g \rangle_{\H} 
	- \bar{\gamma} \langle f , I A_m g \rangle_{\H} 
	- \bar{\delta} \langle f , I B g \rangle_{\H} \\
	=&
	\bar{\alpha} \gamma \langle A_m f , A_m g \rangle_{\dH} 
	+ \bar{\beta} \gamma \langle A_m f , Bg \rangle_{\dH} 
	+ \bar{\alpha} \delta \langle Bf , A_m g \rangle_{\dH} 
	+ \bar{\beta} \delta \langle Bf , Bf \rangle_{\dH} \notag \\
	&- \alpha \bar{\gamma} \langle A_m f , A_m g \rangle_{\dH}
	- \beta \bar{\gamma} \langle Bf , A_m g \rangle_{\dH}
	- \alpha \bar{\delta} \langle A_m f , Bg \rangle_{\dH}
	- \beta \bar{\delta} \langle Bf , Bf \rangle_{\dH} \notag \\
	=& (\bar{\alpha}\gamma - \alpha \bar{\gamma}) \langle A_m f , A_m g \rangle_{\dH}
	+ (\bar{\beta}\delta - \beta \bar{\delta}) \langle B f , B g \rangle_{\dH} \notag \\
	&+\left(\bar{\beta} \gamma - \alpha \bar{\delta} \right) 
	\langle A_m f, B g \rangle_{\dH} 
	- \left(\beta \bar{\gamma} - \bar{\alpha} \delta \right) 
	\langle B f, A_m g \rangle_{\dH} \notag 
	\end{align}
	for $f, g \in D(\HIBCab)$.
	Combining \eqref{eq:1 lem:symmetry} and \eqref{eq:2 lem:symmetry} yields
	\begin{align}
		\langle \HIBCab f, g \rangle_{\dH} - \langle f , \HIBCab g \rangle_{\dH}
		&= (\bar{\alpha}\gamma - \alpha \bar{\gamma}) \langle A_m f , A_m g \rangle_{\dH}
		+ (\bar{\beta}\delta - \beta \bar{\delta}) \langle B f , B g \rangle_{\dH}
		\notag \\
		&+\left(\bar{\beta} \gamma - \alpha \bar{\delta} -1 \right) 
		\langle A_m f, B g \rangle_{\dH} 
		- \left(\beta \bar{\gamma} - \bar{\alpha} \delta -1 \right)
		\langle B f, A_m g \rangle_{\dH},
	\end{align}
	so clearly $\HIBCab$ is symmetric under the given conditions.
\end{proof}

\begin{rem}\label{rem:boundary triple}
The results of this section show that $(\dH,B,A_m)$ is a \emph{quasi boundary triple} for the restriction $(L_m)|_{D(A_m)}$ (see~\autoref{def:qbt}).

	In this context, the identity \eqref{eq:integration by parts} is called the \emph{abstract Green identity}. 
	By \autoref{rem:1}, $A = (A_m)|_{\ker(B)}$ has dense range and $B$ is surjective. This implies that $(A_m,B) \colon D(A_m) \cap D(L_m) \to \dH \times \dH$ has dense range. Finally, $L = (L_m)|_{\ker(B)}$ is a self-adjoint operator on $\H$, by hypothesis. 
\end{rem}

\subsection{Self-adjointness}

In the framework of quasi boundary triples, the symmetric/self-adjoint boundary conditions for $L_m$ have been studied extensively~\cite{derkach1991, derkach1995, BeLa2007, BeMi14, posi08}. 
 	In particular, this applies to the Robin-type conditions $\alpha A f+ \beta Bf =0$, that correspond to the choice $I=0$ for $\HIBCab$.

In this section we study the self-adjointness of $\HIBCab$ in relation to these Robin-type operators and provide formulas for its resolvent.
Throughout, we assume that the parameters $\alpha, \beta, \gamma, \delta$ satisfy the symmetry condition of~\autoref{lem:symmetry}
\begin{equation}
 \bar{\alpha} \gamma,\, \bar{\beta} \delta \in \R, \text{ and } \beta \bar{\gamma} - \bar{\alpha} \delta = 1.
\end{equation}

We begin by introducing the usual Robin-type operators.
\begin{defn}
	For $\alpha, \beta \in \C$ we denote by $L_{\alpha, \beta} \colon D(L_{\alpha,\beta}) \subset \H \to \H$ the \emph{abstract operator with Robin boundary conditions}
	\begin{equation}
	L_{\alpha,\beta} f := L_m f, \quad D(L_{\alpha,\beta}) := \{ f \in D(L_m) \cap D(A_m) \colon \alpha A_m f + \beta B f = 0 \} . 
	\end{equation} 
\end{defn}

We will study the relationship between $L_{\alpha, \beta}$ and $\HIBCab$. Expressing properties of $\HIBCab$ by those of $L_{\alpha, \beta}$ is useful, since the latter are 
better understood. The operator $L_{0,1}=L$ is usually particularly simple.

The operators $L_{\alpha, \beta}$ are symmetric if $\alpha \overline{\beta} \in \R$, which  is implied by the symmetry conditions for $\HIBCab$.

In the Moshinsky-Yafaev model (\autoref{exa:yaf}), the operators $L_{\alpha, \beta}$ correspond to the Laplacian in $\R^3$ with a $\delta$-potential at $x=0$ and coupling (scattering length) $\alpha \beta^{-1}=\alpha \bar{\beta} |\beta|^{-2}\in\R\cup \{\pm\infty\}$.
The relationship between $L_{\alpha,\beta}$ and $\HIBCab$ will be expressed using the following operators that generalise $G_\lambda$.

\begin{defn}
	Let $\bar{\lambda} \in \rho(L_{\alpha,\beta})$. We define the \emph{abstract Dirichlet-operators} associated with $\alpha,\beta$ and $\lambda$ by
	\begin{equation}
	 \Glambdaab = \Big((\gamma A_m + \delta B)R(\overline{\lambda},L_{\alpha,\beta})\Big)^*, \qquad D(\Glambdaab)=\rg(\alpha T_{\bar\lambda}^* + \beta).
	 \end{equation}
	Moreover we define the \emph{abstract Dirichlet-to-Neumann operator} associated with $\alpha,\beta$ and $\lambda$ by
	\begin{equation}
		T_\lambda^{\alpha,\beta} := (\gamma A_m+\delta B) \Glambdaab , \qquad
	D(T_\lambda^{\alpha,\beta}) = \rg(\alpha T_{\bar\lambda}^* + \beta).
	\end{equation}
\end{defn}

In order to investigate these operators, we need the following well-known resolvent formula for $L_{\alpha,\beta}$ (see e.g.~\cite{derkach1991, derkach1995, BeLa2007}).

\begin{lem}\label{lem:resLab}
 Let $(\alpha, \beta)\neq 0$ and $\lambda\in \rho(L)$.
 Then $\lambda \in \rho(L_{\alpha, \beta})$ if and only if $\alpha T_\lambda + \beta$ is one-to-one and $\rg(A)\subset \rg \left(\alpha T_\lambda + \beta\right)$. In this case the resolvent satisfies
 \begin{equation*}
  R(\lambda, L_{\alpha, \beta})=\left(1-\alpha G_\lambda \left(\alpha T_\lambda + \beta\right)^{-1}A\right)R(\lambda, L).
 \end{equation*}
 \end{lem}

\begin{proof}
 Since $L$ is self-adjoint, we have $\lambda, \bar\lambda\in \rho(L)$  and we can write $f=f_0+G_\lambda\varphi$ with $f_0\in D(L)$. The equation $(\lambda - L_{\alpha, \beta})f=g$ then takes the form 
 \begin{equation}\label{eq:sysLab}
 \begin{aligned}
  (\lambda - L_m)f=(\lambda - L)f_0\stackrel{!}{=}&g \\
  \left(\alpha A + \beta B\right) f= \alpha Af_0 + (\alpha T_\lambda + \beta)\varphi\stackrel{!}{=}&0.
 \end{aligned}
 \end{equation}
A solution $\varphi$ to the second equation is clearly unique if and only if $\ker(\alpha T_\lambda + \beta)= \{0\}$, so $(\alpha T_\lambda + \beta)$ must be one-to-one.

Solving the first equation for $f_0=R(\lambda, L)g$, we see that $f_0$ can be any element of $D(L)$, depending on $g$. Hence the solution to the system~\eqref{eq:sysLab} exists for every $g\in \H$ exactly if $\rg(A)\subset \rg \left(\alpha T_\lambda + \beta\right)$.
Under these hypothesis and using the algebraic inverse 
\begin{equation}
 \left(\alpha T_\lambda + \beta\right)^{-1}:\rg(\alpha T_\lambda + \beta)\to D(T_\lambda),
\end{equation}
we obtain the solution to~\eqref{eq:sysLab} as
\begin{equation}
 f_0=R(\lambda, L)g, \qquad \varphi=-\alpha \left(\alpha T_\lambda + \beta\right)^{-1} Af_0,
\end{equation}
which gives the resolvent formula. 
\end{proof}

Similar to \autoref{prop:rechenregeln} we now obtain the following statements. 

\begin{prop}\label{prop:rechenregeln ab}
 For $\lambda\in \rho(L_{\alpha, \beta})\cap\rho(L)$:
 \begin{enumerate}[(i)]
  \item $G_{\bar\lambda}^{\alpha, \beta}$ is densely defined and bounded;
  \item
  The operator $\alpha T_{\lambda}^* + \beta$ has a densely defined inverse 
  \begin{equation*}
   \left(\alpha T_{\lambda}^* + \beta\right)^{-1}:\rg(\alpha T_{\lambda}^* + \beta) \to \dH
  \end{equation*}
and we have
  \begin{equation*}
   G_{\bar\lambda}^{\alpha, \beta}=G_{\bar\lambda} \left(\alpha T_{\lambda}^* + \beta\right)^{-1};
  \end{equation*}
  \item The image satisfies $\rg(G_{\bar\lambda}^{\alpha, \beta}) \subset \ker(\bar\lambda - L_m)\cap D(A_m)$;
\item The following identity holds 
		\begin{equation*}
			(\alpha A_m + \beta B) G_{\bar\lambda}^{\alpha, \beta} = \Id_{D(G_{\bar\lambda}^{\alpha, \beta})}.
		\end{equation*}
\item The operator $T_{\bar\lambda}^{\alpha, \beta}$ is densely defined and given by the formula 
\begin{equation*}
 T_{\bar\lambda}^{\alpha, \beta}=(\gamma T_{\bar\lambda} + \delta)(\alpha T_{\lambda}^* + \beta)^{-1}.
\end{equation*}
 \end{enumerate}
\end{prop}
\begin{proof}
 \begin{enumerate}[(i)]
  \item By definition, $G_{\bar\lambda}^{\alpha, \beta}$ is (a restriction of) the adjoint of an everywhere-defined operator. 
  It is thus sufficient to prove that $G_{\bar\lambda}^{\alpha, \beta}$ is densely defined, because this implies that it is the adjoint of a closable operator, and this is closed and bounded since it is everywhere defined. The claim will thus follow from (ii).
  \item First note that for $\alpha\bar\beta\in \R$ (which follows from the assumed relations of $\alpha, \beta, \gamma, \delta$), we have $L_{\alpha, \beta}=L_{\bar\alpha, \bar \beta}$.
  With the resolvent formula of~\autoref{lem:resLab} we thus have
  \begin{align}
   (\gamma A_m + \delta B) R(\lambda, L_{\alpha, \beta})=&(\gamma A_m + \delta B)\left(1-\bar\alpha G_\lambda \left(\bar\alpha T_\lambda + \bar\beta\right)^{-1}A\right)R(\lambda, L) \notag\\
   =&\left(\gamma A - \bar\alpha (\gamma T_\lambda + \delta) \left(\bar\alpha T_\lambda + \bar\beta\right)^{-1}A\right)R(\lambda, L).\label{eq:Gab form1}
  \end{align}
Now on $\rg(A)$
\begin{equation}
 \bar\alpha (\gamma T_\lambda + \delta) \left(\bar\alpha T_\lambda + \bar\beta\right)^{-1}
 = \gamma \Id_{\rg(A)} + \underbrace{(\bar \alpha \delta- \gamma \bar \beta)}_{=-1}\left(\bar\alpha T_\lambda + \bar\beta\right)^{-1},  
\end{equation}
so~\eqref{eq:Gab form1} simplifies to
\begin{equation}
 (\gamma A_m + \delta B) R(\lambda, L_{\alpha, \beta})=\left(\bar\alpha T_\lambda + \bar\beta\right)^{-1} A R(\lambda, L).
 \label{eq:3.20}
\end{equation}
This shows that, for all $\varphi$ in the domain of the adjoint of $\left(\bar\alpha T_\lambda + \bar\beta\right)^{-1}$ -- which is well defined since $\rg(A)\subset D(\left(\bar\alpha T_\lambda + \bar\beta\right)^{-1})$ is dense by~\autoref{rem:1}, we have
\begin{equation}
  \left((\gamma A_m + \delta B) R(\lambda, L_{\alpha, \beta})\right)^*\varphi=G_{\bar\lambda} \left(\left(\bar\alpha T_{\lambda} + \bar\beta\right)^{-1}\right)^*\varphi
\end{equation}
We now need to show that $D(G_{\bar\lambda}^{\alpha, \beta})=\rg(\alpha T_\lambda^*+\beta)$ is contained in the domain of this adjoint and dense.
Density is an immediate consequence of~\autoref{lem:resLab}, since 
\begin{equation}
 \rg(\alpha T_\lambda^*+\beta)^\perp = \ker(\bar \alpha T_\lambda + \bar \beta)=\{0\}.
\end{equation}
For all $\psi \in D(T_\lambda^*)$, $\varphi \in \rg(\bar \alpha T_\lambda + \bar \beta)$ we have
\begin{equation}
 \langle (\alpha T_\lambda^*+\beta)\psi, \left(\bar\alpha T_\lambda + \bar\beta\right)^{-1} \varphi \rangle_{\dH}
 = \langle \psi, \varphi \rangle_{\dH},
\end{equation}
so we clearly have
\begin{equation}
 \left(\left(\bar\alpha T_\lambda + \bar\beta\right)^{-1}\right)^*(\alpha T_\lambda^*+\beta)=\Id_{D(T_\lambda^*)}.
\end{equation}
This completes the proof of (ii) and thereby also (i).
\item The fact that $\rg(G_{\bar\lambda}^{\alpha, \beta}) \subset \ker(\bar\lambda - L_m)$ is immediate from (ii).
Since the range of $(\alpha T_\lambda^*+\beta)^{-1}$ is contained in $D(T_\lambda^*)$, we also have $\rg(G_{\bar\lambda}^{\alpha, \beta}) \subset  D(A_m)$.
\item Again using (ii) we find 
\begin{equation*}
 (\alpha A_m +  \beta B) G_{\bar\lambda}^{\alpha, \beta}
 = (\alpha A_m +  \beta B)G_{\bar\lambda} \left(\alpha T_{\lambda}^* + \beta\right)^{-1}
 = (\alpha T_{\bar\lambda} + \beta)\left(\alpha T_{\lambda}^* + \beta\right)^{-1}.
\end{equation*}
Since $T_{\bar\lambda} \subset T_\lambda^*$ this proves the claim. 
\item This follows immediately from (i) and (ii).
\qedhere
 \end{enumerate}

\end{proof}

We can now go back to investigating the operator $\HIBCab$.
The following lemma provides a parametrisation of $D(\HIBCab)$ in terms of $D(L_{\alpha,\beta})$, under the condition that $\Id-G^{\alpha, \beta}_\lambda I^*$ is invertible for some $\lambda \in \rho(L_{\alpha,\beta})$.
This is certainly satisfied if $L_{\alpha,\beta}$ is self-adjoint, $A_m$ is infinitesimally $L_{\alpha,\beta}$-bounded and the imaginary part of $\lambda$ is large enough. It is also usually satisfied  if there is a hierarchical structure of the form that we have in applications to quantum field theory, see~\autoref{rem:hierarchy}. 
In the Moshinsky-Yafaev model (\autoref{exa:yaf}), this is particularly obvious, since there $(G^{\alpha, \beta}_\lambda I^*)^2=0$, so the inverse is simply given by $(\Id- G_\lambda^{\alpha, \beta} I^*)^{-1}=\Id+G_\lambda^{\alpha, \beta} I^*$.

This parametrisation for the case $\alpha=0$ appears already in the works~\cite{LaSch19,La19,IBCrelat, IBCmassless}, where it plays an important role.

\begin{lem}\label{lem:domains} 
Assume that $\bar \lambda\in \rho(L)\cap \rho(L_{\alpha, \beta})$.
	If $1 \in \rho(G^{\alpha,\beta}_\lambda I^*)$, we denote $\Gamma^{\alpha, \beta}_\lambda:=R(1, \Glambdaab I^*)$ and
	the equality
	\begin{equation}
	D(\HIBCab) = (\Id - G^{\alpha,\beta}_\lambda I^*)^{-1} D(L_{\alpha,\beta})= \Gamma^{\alpha, \beta}_\lambda D(L_{\alpha,\beta}).
	\end{equation}
	holds.
\end{lem} 
\begin{proof}
	Since both sides are subsets of $D(L_m) \cap D(A_m)$ it is sufficient to verify the boundary conditions. 
	
	\smallskip 
	
	Assume first that $f = f_0 + G_{\bar\lambda} \varphi \in D(\HIBCab)$, $f_0\in D(L)$. Using the interior-boundary condition and~\autoref{lem:resLab} we first find 
	\begin{equation}\label{eq:I*DIBC}
	 I^*f = (\alpha T_{\bar\lambda} + \beta)\varphi + Af_0 \in \rg( \alpha T_{\bar\lambda} + \beta) \subset \rg(T_\lambda^*+\beta).
	\end{equation}
We thus have $f\in D(G_\lambda^{\alpha,\beta} I^*)$ and can use~\autoref{prop:rechenregeln ab} (iv) to obtain
	\begin{equation}
	(\alpha A_m + \beta B)(\Id - G^{\alpha,\beta}_\lambda I^*) f 
	= (\alpha A_m + \beta B) f - (\alpha A_m + \beta B) G^{\alpha,\beta}_\lambda I^* f = I^* f - I^* f = 0 . 
	\end{equation}
	
	\smallskip  
	
	Conversely, we assume that $\eta \in D(L_{\alpha,\beta})$. Since $(\Id - G^{\alpha,\beta}_\lambda I^*)$ is invertible and hence surjective, there exists an $f\in D(G_\lambda^{\alpha,\beta} I^*)$ with $\eta = (\Id - G^{\alpha,\beta}_\lambda I^*) f$. Note that
	\begin{equation}
	f = \underbrace{\eta}_{\in D(A_m) \cap D(B)} - \underbrace{G^{\alpha,\beta}_\lambda I^* f}_{\in D(A_m) \cap D(B)} \in D(A_m) \cap D(B) .
	\end{equation}
	It follows from~\autoref{prop:rechenregeln ab} (iv) that
	\begin{equation}
	(\alpha A_m + \beta B) f = (\alpha A_m + \beta B) G^{\alpha,\beta}_\lambda I^* f
	= I^* f 
	\end{equation}
	and hence $f \in D(\HIBCab)$.
\end{proof}

The following lemma relates relative bounds of $A_m$ and $T_\lambda$, using the $L$-boundedness of $A$ and the decomposition of~\autoref{prop:rechenregeln}(i). Recall that by our convention 
\begin{equation}
D(I T_{\lambda}^{\alpha,\beta} I^*) = \{f \in \H: I^*f\in D( T_{\lambda}^{\alpha,\beta} )\}. 
\end{equation}

\begin{lem}\label{lem:A rel HIBC-bounded}
	Assume that $\lambda, \bar\lambda \in \rho(L_{\alpha,\beta})\cap \rho(L)$ and $1 \in \rho(\Glambdaab I^*)$.
	Further, assume that $I T_{\lambda}^{\alpha,\beta} I^*$ is relatively $(\Id - G_{\bar{\lambda}} I^*)^* (L_{\alpha,\beta}-\lambda) (\Id - \Glambdaab I^*)$-bounded of bound $a < 1$.
	
	Then
	\begin{enumerate}[(i)]
		\item 
		$I T_\lambda^{\alpha,\beta} I^*$ is relatively $\HIBCab$-bounded.
		If $a=0$, i.~e.~the bound relative to $(\Id - G_{\bar{\lambda}} I^*)^* (L_{\alpha,\beta}-\lambda) (\Id - \Glambdaab I^*)$ is infinitesimal, then the $\HIBCab$-bound is also infinitesimal.
		\item $I A_m$ is relatively $\HIBCab$-bounded. 
		If $a = 0$, then the $\HIBCab$-bound is also infinitesimal.
	\end{enumerate}
	
\end{lem}

\begin{proof}
\begin{enumerate}[(i)]
\item
	From the definition of $G_{\bar{\lambda}}^{\alpha,\beta}$ we obtain
	\begin{equation}\label{eq:G^ab L}
	(G_{\bar{\lambda}}^{\alpha,\beta} I^*)^*({\lambda} - L_{\alpha,\beta})
	= I (G_{\bar{\lambda}}^{\alpha,\beta})^* ({\lambda}-L_{\alpha,\beta})
	= \gamma I A_m + \delta I B . 
	\end{equation}
	Using $\rg(\Glambdaab) \subset \ker(\lambda - L_m)$ and \autoref{lem:domains} it follows for $f \in D(\HIBCab)$
	\begin{align}
	(\HIBCab - \lambda) f 
	=& (L_{\alpha,\beta} - \lambda)(\Id - G_{\lambda}^{\alpha,\beta} I^*)f + \gamma A_m f + \delta B f \notag \\
	=& (\Id - G_{\bar{\lambda}}^{\alpha,\beta} I^*)^*(L_{\alpha,\beta} - {\lambda})(\Id - G_{{\lambda}}^{\alpha,\beta} I^*)f \notag \\ 
	&+ (G_{\bar{\lambda}}^{\alpha,\beta} I^*)^*(L_{\alpha,\beta} - {\lambda})(\Id - G_{\lambda}^{\alpha,\beta} I^*)f + \gamma I A_m f + \delta I B f .
	\end{align}
	With~\eqref{eq:G^ab L} the last line becomes 
\begin{equation}
	 (G_{\bar{\lambda}}^{\alpha,\beta} I^*)^*(L_{\alpha,\beta} - {\lambda})(\Id - G_{\lambda}^{\alpha,\beta} I^*)f + \gamma I A_m f + \delta I B f
	 =  (\gamma I A_m + \delta I B) G_{\lambda}^{\alpha,\beta}I^*f = I T_{\lambda}^{\alpha, \beta} I^*,
\end{equation}
and consequently 
\begin{equation}\label{eq:perturbation}
 (\HIBCab - \lambda)f = (\Id - G_{\bar{\lambda}}^{\alpha,\beta} I^*)^*(L_{\alpha,\beta} - {\lambda})(\Id - G_{{\lambda}}^{\alpha,\beta} I^*)f + I T_{\lambda}^{\alpha, \beta} I^*.
\end{equation}
	In particular we obtain that $D(\HIBCab) \subset D(I T^{\alpha,\beta}_\lambda I^*)$. 
	Since $I \in \mathcal{L}(\dH,\H)$ we conclude
	\begin{align}
	\| I T^{\alpha,\beta}_\lambda I^* f \| 
	&\leq a \| (\Id - G_{\bar{\lambda}}^{\alpha,\beta} I^*) (L_{\alpha,\beta}-\lambda) (\Id - \Glambdaab I^*) f \| + b \| f \| \\
	&\leq a \| \HIBCab f \| + a    \| I T^{\alpha,\beta}_\lambda I^* f \| + (b + |\lambda|) \| f \|
	\notag 
	\end{align}
	for all $f \in D(\HIBCab)$. Since $a   < 1$, the claim follows by absorbing the $\| I T_{\lambda}^{\alpha,\beta} I^* f\|$-term on the left hand side. 
\item
By \autoref{prop:rechenregeln ab}(i) we have that $G_{\bar{\lambda}}^{\alpha,\beta}$ is bounded and by the proof of \autoref{prop:rechenregeln ab}(i) that $(\gamma A_m + \delta B)R(\lambda,L_{\alpha,\beta})$ is bounded. It follows from \eqref{eq:3.20} that
\begin{equation}
	(\bar{\alpha}T_\lambda + \bar{\beta})^{-1}
	A R(\lambda,L)
	= (\gamma A_m + \delta B)R(\lambda,L_{\alpha,\beta}) .
\end{equation}
Hence \autoref{lem:resLab} implies that
\begin{align}
	B R(\lambda,L_{\alpha,\beta})
	&= BR(\lambda.L_{\bar{\alpha},\bar{\beta}}) \\
	&= B (\Id - \bar{\alpha}\Glambda (\bar{\alpha}T_\lambda + \bar{\beta})^{-1}A)R(\lambda, L) \notag \\
	&= - \bar{\alpha} (\bar{\alpha}T_\lambda + \bar{\beta})^{-1}
	A R(\lambda,L) \notag \\
	&= - \bar{\alpha} (\gamma A_m + \delta B)R(\lambda,L_{\alpha,\beta})\notag
\end{align}
is bounded, and we conclude that $\gamma A_m R(\lambda,L_{\alpha,\beta})$ is also bounded.
In the following we consider the case $\beta \not = 0$. The case $\beta = 0$ works by the same arguments. 

By \autoref{lem:domains} we obtain, using the Robin boundary condition
 	\begin{align}
 		\gamma I A_m f + \delta I Bf 
 		&= I(\gamma A_m + \delta B) \Glambdaab I^* f + I( \gamma A_m + \delta B) (\Id - \Glambdaab I^*) f \\
 		&= I T_{\lambda}^{\alpha,\beta} I^* f 
 		+ I( \gamma A_m + \delta B) (\Id - \Glambdaab I^*) f \notag \\
 		&= I T_{\lambda}^{\alpha,\beta} I^* f 
 		+ \frac{1}{\bar\beta} I A_m (\Id - \Glambdaab I^*) f
 		\notag 
 	\end{align}
for $f \in D(\HIBCab)$. The first term is relatively $\HIBCab$-bounded by (i). Since $A_mR(\lambda,L_{\alpha,\beta})$ is bounded and $1 \in \rho(\Glambdaab I^*)$, we obtain using \eqref{eq:perturbation} that
 	\begin{align}
 		\| I A_m (\Id - \Glambdaab I^*) f \|
 		&\leq \bar{a} \cdot \| (L_{\alpha,\beta}-\lambda) (\Id - \Glambdaab I^*) f \| + b \| f \| \\
 		&\leq \bar{a} C \cdot \|(\Id - \Glambdaab I^*)^* (L_{\alpha,\beta}-\lambda) (\Id - \Glambdaab I^*) \|  + b \| f \| \notag \\
 		&= \bar{a} C \cdot \| \HIBCab f \| + \bar{a} C \cdot 
 		\| I T_\lambda^{\alpha,\beta} I^* f \| + b \| f \| \notag
 	\end{align}
 	for $f \in D(\HIBCab)$, where $C \coloneqq \| (\Id - \Glambdaab I^*)^{-1} \|$. Using (i), we conclude 
 	\begin{equation}
 		\| \gamma I A_m f + \delta I B f \|
 		\leq \tilde{a} \| \HIBCab f \| + b \| f \| 
 	\end{equation}
 	for $f \in D(\HIBCab)$, i.e. the operator $(\gamma I A_m + \delta I B)$ is relatively $\HIBCab$-bounded of bound $\tilde{a}~:=a\cdot\left(1+\frac{\bar{a} \cdot C}{|\beta|}\right)$. In particular the bound is infinitesimal if the bound $a$ is infinitesimal. 

 	Now the claim follows, since by the IBC
 	\begin{equation}
 		\gamma I A_m f + \delta I B f
 		= \frac{1}{\bar{\beta}} A_m f - \frac{\delta}{\bar\beta} I I^* f,
 	\end{equation}
 	and $II^*$ is bounded. 

 	\qedhere  
\end{enumerate}
\end{proof}

Apart from the statement of~\autoref{lem:A rel HIBC-bounded}, an important finding is the equation~\eqref{eq:perturbation}. It represents $\HIBCab$ as a perturbation of an operator that is obtained by transforming $L_{\alpha, \beta}$.
This leads to the main theorem of this section.

\begin{thm}\label{thm:HIBC self-adjoint}
Assume that $L_{\alpha,\beta}$ is self-adjoint and let $\lambda \in \rho(L_{\alpha,\beta}) \cap \rho(L)$. Assume also that $1 \in \rho(\Glambdaab I^*)\cap \rho(G_{\bar{\lambda}}^{\alpha, \beta}I^*)$ and
$I T_\mu^{\alpha,\beta} I^*$ is relatively $(\Id - G_{\bar{\mu}}^{\alpha,\beta} I^*)^*(L_{\alpha,\beta} - \mu)(\Id - G_{\mu}^{\alpha,\beta} I^*)$ bounded with bound $a < 1$ 
for $\mu \in \{\lambda,\bar{\lambda}\}$.	
Then $\HIBCab$ is self-adjoint.

	Moreover, with $\Gamma^{\alpha, \beta}_\lambda=R(1, G_\lambda^{\alpha, \beta}I^\ast)$, we have 
	\begin{align*}
	 \lambda\in \rho(L_{\alpha,\beta})\cap \rho(\HIBCab) \quad 
	 \Longleftrightarrow \quad 1\in \rho\Big((\Gamma_{\bar \lambda}^{\alpha, \beta})^* I T_{\lambda}^{\alpha, \beta} I^* 
	\Gamma_{\lambda}^{\alpha, \beta}R(\lambda, L_{\alpha, \beta}\Big)
	\end{align*}
and the resolvent is then given by
	\begin{align*}
	R(\lambda, \HIBCab) 
	=& \Gamma_{\lambda}^{\alpha, \beta}R(\lambda, L_{\alpha, \beta}) \Big( 1 - (\Gamma_{\bar \lambda}^{\alpha, \beta})^* I T_{\lambda}^{\alpha, \beta} I^* 
	\Gamma_{\lambda}^{\alpha, \beta}R(\lambda, L_{\alpha, \beta})\Big)^{-1} 
	(\Gamma_{\bar \lambda}^{\alpha, \beta})^*.\notag
	\end{align*}
\end{thm} 
\begin{proof}
As $L, L_{\alpha, \beta}$ are self-adjoint, we also have $\bar\lambda\in \rho(L_{\alpha,\beta}) \cap \rho(L)$.

Using~\eqref{eq:perturbation} twice, we write
\begin{align}
 \HIBCab = & \tfrac12(\HIBCab-\lambda) + \tfrac12(\HIBCab-\bar\lambda) + \mathrm{Re}(\lambda)\\
 = &  \tfrac12 (\Id - G_{\bar{\lambda}}^{\alpha,\beta} I^*)^*(L_{\alpha,\beta} - {\lambda})(\Id - G_{{\lambda}}^{\alpha,\beta} I^*)f 
 +\tfrac12 (\Id - G_{\lambda}^{\alpha,\beta} I^*)^*(L_{\alpha,\beta} - {\bar\lambda})(\Id - G_{\bar\lambda}^{\alpha,\beta} I^*)f 
 \notag \\
 & + \tfrac12( I T_{\bar\lambda}^{\alpha, \beta} I^* + I T_{\lambda}^{\alpha, \beta} I^*)+  \mathrm{Re}(\lambda).\notag
\end{align}
Since both $\HIBCab$ and the sum of the two expressions involving $L_{\alpha, \beta}$ are symmetric on $D(\HIBCab)$, so is the last line. Self-adjointness of $\HIBCab$ thus follows from the Kato-Rellich theorem.

To show the resolvent formula, we take~\eqref{eq:perturbation} and use that $\lambda\in \rho(L_{\alpha, \beta})$ to write
\begin{align}
 \lambda-\HIBCab= &(\Id - G_{\bar{\lambda}}^{\alpha,\beta} I^*)^*\left(\lambda - L_{\alpha,\beta}  - (\Gamma_{\bar\lambda}^{\alpha, \beta})^* I T_{\lambda}^{\alpha, \beta} I^*\Gamma_{\lambda}^{\alpha, \beta}\right)(\Id - G_{{\lambda}}^{\alpha,\beta} I^*) \notag \\
 =& (\Id - G_{\bar{\lambda}}^{\alpha,\beta} I^*)^*\left(1 - (\Gamma_{\bar\lambda}^{\alpha, \beta})^* I T_{\lambda}^{\alpha, \beta} I^*\Gamma_{\lambda}^{\alpha, \beta}R(\lambda, L_{\alpha, \beta})\right)(\lambda-L_{\alpha, \beta})(\Id - G_{{\lambda}}^{\alpha,\beta} I^*).
\end{align}
Since $1 \in \rho(\Glambdaab I^*)\cap \rho(G_{\bar\lambda}^{\alpha, \beta}I^*)$ and $\lambda\in \rho(L_{\alpha,\beta})$, the right hand side is invertible if and only if $1\in \rho((\Gamma_{\bar \lambda}^{\alpha, \beta})^* I T_{\lambda}^{\alpha, \beta} I^*\Gamma_{\lambda}^{\alpha, \beta}R(\lambda, L_{\alpha, \beta})$.
Assuming this implies the formula as claimed.
\end{proof} 

In~\autoref{exa:yaf}, the hypothesis on $T_\lambda^{\alpha, \beta}$ are all trivially satisfied, since $\dH$ is one-dimensional. 
For the applications in~\cite{LaSch19,La19,IBCrelat, IBCmassless}, proving the relative bound for $T_\lambda$ was the main technical difficulty. For the case $\alpha = 0$ relevant there, we can formulate the following corollary.
A similar abstract formulation has appeared in~\cite{posi20}.

\begin{cor}\label{cor:HIBC self-adjoint alpha = 0}
	Let $\lambda\in \rho(L)\cap\R$ and assume that $1 \in \rho(\Glambda I^*)$ and $I T_\lambda I^*$ is relatively $(\Id - G_{\lambda} I^*)^* (L-\lambda) (\Id - \Glambda I^*)$-bounded of bound $a < |\beta|^{-2}$.
	Then $\HIBCb$ is self-adjoint. 
\end{cor} 
\begin{proof}
 The symmetry condition $\beta \bar{\gamma} - \bar{\alpha} \delta = 1$ with $\alpha=0$ implies that $\bar\gamma=\beta^{-1}$.
 Then with~\autoref{prop:rechenregeln ab}
 \begin{equation*}
  T^{0, \beta}_\lambda= (\gamma  A_m + B) G^{0, \beta}_\lambda= ({\bar\beta}^{-1} A_m + \delta) G_\lambda \beta^{-1} = |\beta|^{-2} T_\lambda + \delta \beta^{-1}.
 \end{equation*}
The claim thus follows from~\autoref{thm:HIBC self-adjoint}.
\end{proof}

For $\alpha \not = 0$ we obtain the following corollary, which highlights a key difference, namely that for $\alpha\neq 0$ the boundary condition may be used to control $A_m$.  

\begin{cor}\label{cor:HIBC s-a alpha n= 0}
 Let $\lambda \in \rho(L)\cap \R$. If $-\beta \in \rho(\alpha T_{{\lambda}})$ then $\lambda\in \rho(L_{\alpha, \beta})$, and if additionally $1 \in \rho(\Glambdaab I^*)$ then  $\HIBCab$ is self-adjoint.
\end{cor}
\begin{proof}
 By~\autoref{lem:resLab}, $-\beta \in \rho(\alpha T_{{\lambda}})$ implies that $\lambda\in \rho(L_{\alpha, \beta})\cap \R$, so $L_{\alpha, \beta}$ is self-adjoint.
 Now assume that $1 \in \rho(\Glambdaab I^*)$ (note that this holds if $A$ is infinitesimally $L$-bounded and $\mathrm{dist}(\lambda, \sigma(L))$ is large enough by~\autoref{prop:rechenregeln ab} (ii)). Since $\alpha T_\lambda +\beta$ has a bounded inverse, we have $\rg(T_{\lambda}^\ast + \beta)=\dH$ and
 \begin{align}
			\alpha T_{{\lambda}}^{\alpha,\beta} 
			&= \alpha ({\gamma} T_{\lambda} + {\delta}) (\alpha T_{\lambda}^\ast + \beta)^{-1} \label{eq:4.25}
			 \\		
 			&= \gamma \cdot \Id + (\beta \gamma - \alpha \delta) R(-\beta, \alpha T_{{\lambda}}) \notag \\
			&= \gamma \cdot \Id + R(-\beta, \alpha T_{{\lambda}}).
			\notag 
\end{align}
The operator $T_{{\lambda}}^{\alpha,\beta}$ is thus bounded, and the hypothesis of~\autoref{thm:HIBC self-adjoint} are satisfied.
 \end{proof}

\begin{rem}\label{rem:posi}
 Posilicano~\cite{posi20} discusses self-adjointness of the operator $\HIBC$ (with $\H=\dH$ and $I=\Id$), considering $T=T_{z_0}$ (for some fixed $z_0\in \C$) as a parameter.
 The resolvent of $\HIBC$ is constructed by first perturbing $L=L_{0,1}$ to obtain $L_{1,0}$ as in~\autoref{lem:resLab} and then obtaining  $\HIBC$ as an  extension of the restriction of $L_m + IA_m$ to $D(L_{1,0})\cap \ker(I^*-B)$, which is also a restriction of $L_{1,0}$.
 
 In our notation, the formula fo the resolvent reads, with $\widehat G_z=\big((I^*-B)R(\bar z, L_{1,0})\big)^*$ (c.f.~\cite[Thm.3.4]{posi20})
 \begin{align}\label{eq:res posi}
  R(\lambda, \HIBC) = & R(\lambda, L_{1,0}) - \widehat G_\lambda \big((I^*-B)\widehat G_\lambda\big)^{-1} (I^*-B)R(\lambda, L_{1,0}) \\
  = &\Big(1- \widehat G_\lambda \big((I^*-B)\widehat G_\lambda\big)^{-1} (I^*-B)\Big)\big(1-G_\lambda T_\lambda^{-1} A \big) R(\lambda, L).\notag
 \end{align}
The validity of this formula requires somewhat stronger hypothesis than \autoref{cor:HIBC self-adjoint alpha = 0}, such as invertibility of $T_\lambda$, though one can obtain a formula as in \autoref{thm:HIBC self-adjoint} by expanding~\eqref{eq:res posi} and thereby recover the weeker hypothesis (for $\alpha=0$), see~\cite[Thm.3.10]{posi20}.
\end{rem}

\section{Classification of interior-boundary conditions}\label{sect:ext}

In this section we will embed the IBC-operators studied in the previous sections into the extension theory of symmetric operators to obtain general criteria for self-adjointness and a classification of symmetric and self-adjoint IBCs. 
To achieve this, we take a family of self-adjoint IBC-operators that are all extensions of a common symmetric operator and thus all restrictions of one operator. We then construct a quasi boundary triple for such a ``maximal'' operator and thereby obtain conditions for a generalised IBC to be symmetric or self-adjoint.

Consider for $0\neq g\in \R$ the domain
\begin{equation}
D(H_0) = \{ f \in D(L_m) \cap D(A_m) \colon gA f = gBf = I^*f \}= D(H_\mathrm{IBC}^{0,g})\cap D(H_\mathrm{IBC}^{g,0}).
\end{equation}
Clearly we have $D(H_0)\subset D(\HIBCab)$ if $\alpha+\beta=g$. Furthermore, we have
\begin{equation}
\HIBCab\vert_{D(H_0)} = L_m + (\gamma+\delta)g II^*,
\end{equation}
so the actions of all $\HIBCab$ with $(\gamma+\delta)=\mathrm{const.}$ agree on $D(H_0)$ and all of these operators are 
symmetric/self-adjoint extensions of $H_0:=H_\mathrm{IBC}^{0,g}\vert_{D(H_0)}$. 
We now restrict ourselves to the case $\alpha+\beta=1=\gamma+\delta$. More general conditions can be reduced to this case by modifying the operator $I$, see~\autoref{rem:a+b=1}.

Throughout this section we assume, in addition to the global notation and~\autoref{ass:general} the following, which includes the hypothesis of~\autoref{thm:HIBC self-adjoint} with $\alpha=0, \beta=1$.
\begin{asu}\label{ass:prop 4.6}
	Suppose that there exists $\lambda\in \C$ with
	\begin{enumerate}[(i)]
        \item $\lambda \in \rho(\HIBC)  \cap \rho(L)$;
		\item 
		$1 \in \rho(\Glambda I^*)\cap \rho(G_{\bar{\lambda}}I^*)$;
				\item
		$I T_\mu I^*$ is relatively $(\Id - G_{\bar{\mu}} I^*)^*(L_ - \mu)(\Id - G_{\mu} I^*)$ bounded
		with 
		bound $a < 1$
		for $\mu \in \{\lambda,\bar{\lambda}\}$;
		\item $(1-I^*G_\lambda)^{-1}$ leaves $D(T)$ invariant.
	\end{enumerate}
\end{asu}

\begin{defn}
	We define the operator $H_0 \colon D(H_0) \subset \H\to\H$ by
	\begin{equation}
	H_0 f = L_m f + II^* f , \qquad D(H_0) = \{ f \in D(L_m) \cap D(A_m) \colon A_m f = Bf = I^*f \} .
	\end{equation} 
\end{defn}

\begin{lem}\label{lem:H_m}
	The operator
	\begin{equation}
	H_m:=L_m+ II^* + I(A_m-B), \qquad D(H_m) =  D(L_m) \cap D(A_m) 
	\label{eq:H_m}
	\end{equation}
	is a restriction of $H_0^*$.
\end{lem}
\begin{proof}
	Take $f\in D(H_m)$ and $g\in D(H_0)$, then by~\autoref{lem:integration by parts}
	\begin{align}
	\langle f, H_0 g \rangle_\H &= \langle L_m f,  g \rangle_{\H}  -\langle Bf, A_m g\rangle_{\dH} + \langle A_mf, Bg \rangle_{\dH} + \langle II^*f, g \rangle_{\H} \\
	&= \langle (L_m+I(A_m-B)+II^*) f,  g \rangle.\notag
	\end{align}
	If $D(H_0)$ is dense, this proves that $H_m\subset H_0^*$ as operators. If $D(H_0)$ is not dense, the adjoint is not a well-defined operator, but the equation shows that the graph of $H_m$ is contained in the adjoint relation to the graph of $H_0$ (see~\hyperref[sect:relations]{Appendix \ref*{sect:relations}}), so $H_m\subset H_0^*$ in the sense of relations.  
\end{proof}

Note that we avoid here the hypothesis that $D(H_0)$ is dense, even though we expect this to be the case in relevant examples, as it might be quite difficult to verify. 

\begin{lem}\label{lem:integration by parts-2}
	We have the abstract Green identity
	\begin{align}
	\langle H_m f , g \rangle_\H - \langle f, H_m g \rangle_{\H} 
	= \langle (B-I^*)f, (A_m-I^*)g\rangle_{\dH} - \langle (A_m-I^*)f, (B-I^*)g\rangle_{\dH}
	\label{eq:integration by parts-2}
	\end{align}
	for $f, g \in D(H_m)$.
\end{lem} 
\begin{proof}
	Using the formula~\eqref{eq:integration by parts} for $L_m$ and~\eqref{eq:H_m}, we find
	\begin{align}
	\langle H_m f , g \rangle_\H - \langle f, H_m g \rangle_{\H} 
	= &\langle Bf,A_mg\rangle_{\dH} - \langle A_mf, Bg\rangle_{\dH} \\
	&+\langle (A_m-B+I^*)f, I^*g \rangle_{\dH} - \langle I^*f, (A_m-B+I^*)g \rangle_{\dH} \notag\\
	=&-\langle (B-I^*)f,I^*g \rangle_{\dH} + \langle Bf,A_mg\rangle_{\dH} - \langle I^*f, A_m g\rangle\notag\\
	&+ \langle A_mf, I^*g\rangle_{\dH}\notag - \langle A_mf, Bg\rangle_{\dH} + \langle I^*f, (B-I^*)g \rangle_{\dH},
	\end{align}
	which yields the formula as claimed.
\end{proof}

We will obtain a classification of the extensions of $H_0$ by constructing a quasi boundary triple for $H_m$.
To this end, we define the corresponding abstract Dirichlet operator. 

\begin{defn}\label{def:F_lambda} 
We define
	\begin{equation}
	F_\lambda := ((A_m-I^*) R(\overline{\lambda}, \HIBC))^* , \qquad 
	D(F_\lambda):=D(T).
	\end{equation}
\end{defn}

The following lemma will help us when analysing $F_\lambda$ in more detail.

\begin{lem}\label{lem:G Gamma}
	We have
	\begin{equation}
	 (\Id - \Glambda I^*)^{-1} \Glambda
	= \Glambda (\Id - I^* \Glambda)^{-1} .
	\label{eq: G Gamma}
	\end{equation}
\end{lem}
\begin{proof}
	The identity
	\begin{equation}
	(\Id - \Glambda I^*) \Glambda
	= \Glambda - \Glambda I^* \Glambda
	= \Glambda (\Id - I^* \Glambda) . 
	\end{equation}
	implies the claim by multiplying with $(\Id - \Glambda I^*)^{-1}$ from the right hand side and $(\Id - I^* \Glambda)^{-1}$ from the left hand side.
\end{proof}

\begin{prop}\label{prop: Omega}
	The operator $F_\lambda$ for $\lambda \in \rho(\HIBC)\cap \rho(L)$ satisfies
	\begin{enumerate}[(i)]
		\item $\rg(F_\lambda) \subseteq \ker(\lambda - H_m)$; 
		\item $(B-I^*)F_\lambda = \Id_{D(T)}$. 
	\end{enumerate}
\end{prop}
\begin{proof}
	We begin by proving that $\rg(F_\lambda) \subseteq D(H_m)$ and then check properties i), ii). From~\autoref{thm:HIBC self-adjoint} we have, denoting $\Gamma_\lambda:=\Gamma_\lambda^{0,1}=(\Id-\Glambda I^\ast)^{-1}$,
	\begin{align}
	R(\overline{\lambda}, \HIBC) 
	=& \Gamma_{\bar\lambda}R(\bar\lambda, L) ( 1 - \Gamma_\lambda^* I T_{\overline{\lambda}}  I^* 
	\Gamma_{\bar\lambda}R(\bar\lambda, L))^{-1} 
	\Gamma_\lambda^*\notag\\
	= &
	\Gamma_{\bar\lambda} R(\bar\lambda,L)\left(1 + \Gamma_\lambda^*I 
	T_{\overline{\lambda}}  I^* \Gamma_{\bar\lambda}R(\bar\lambda, L) ( 1 - \Gamma_\lambda^* I T_{\overline{\lambda}}  I^* 
	\Gamma_{\bar\lambda}R(\bar\lambda, L))^{-1} \right) 
	\Gamma_\lambda^* \notag\\
	=& \Gamma_{\bar\lambda} R(\bar\lambda,L) \Gamma_\lambda^*\left(1 +  I 
	T_{\overline{\lambda}}  I^* R(\bar \lambda, \HIBC)\right).\label{eq:res IBC self}
	\end{align}
	Denote $\Theta_\lambda = (A_m\Gamma_{\bar\lambda} R(\bar\lambda,L) \Gamma_\lambda^*)^*$. We have
	\begin{equation}
	A_m \Gamma_{\bar\lambda} R(\bar\lambda,L) = A_mR(\bar\lambda,L) + A_m G_{\bar \lambda} I^*\Gamma_{\bar\lambda} 
	R(\bar\lambda,L)
	= G_\lambda^* + T_{\bar\lambda} I^*\Gamma_{\bar\lambda} 
	R(\bar\lambda,L),
	\end{equation}
	whence
	\begin{equation}\label{eq:Theta}
 	\Theta_\lambda = \Gamma_\lambda G_\lambda  + \Gamma_\lambda R(\lambda, L)\Gamma_{\bar\lambda}^*I T_{\bar\lambda}^* .
	\end{equation}
	The first term, $\Gamma_\lambda G_\lambda = G_\lambda (1-I^*G_\lambda)^{-1}$ maps $D(T)$ to $D(H_m)$ since by~\autoref{lem:G Gamma} the operator
	$(1-I^*G_\lambda)^{-1}$ leaves $D(T)$ invariant and $G_\lambda$ maps $D(T)$ to $D(H_m)$.
	The second term acts on $D(T)$ as $\Gamma_\lambda R(\lambda, L)\Gamma_{\bar\lambda}^*I T_{\lambda}$ because $T_\lambda\subset T_{\bar\lambda}^*$.
	By~\autoref{lem:domains},
	$\Gamma_\lambda R(\lambda, L) \Gamma_{\bar\lambda}^* I$ is a bounded operator from $\dH$ to $D(\HIBC)\subset D(H_m)$, 
	so $\Theta_\lambda$ maps $D(T)$ to $D(H_m)$. Since $I^*$ maps $D(\HIBC)$ to $D(T)$ and $I^*\Gamma_\lambda G_\lambda= ((1-I^*G_\lambda)^{-1}-1)$ leaves $D(T)$ invariant, by hypothesis, we see that $I^*\Theta_\lambda$ leaves $D(T)$ invariant.
	From~\eqref{eq:res IBC self} we then see that
	\begin{equation}
 	F_\lambda = \Theta_\lambda + R(\lambda, \HIBC)
 	I T_{\lambda} I^* \Theta_\lambda- R(\lambda, \HIBC)I, \label{eq:Omega}
 	\end{equation}
	and thus $\rg(F_\lambda) \subseteq D(H_m)$. 
	
	For (i) it is now sufficient to prove that $\rg(F_\lambda) \subseteq \ker(\lambda- H_0^*)$, by~\autoref{lem:H_m}, which follows from
	\begin{align}
	\langle F_\lambda \varphi , (H_0-\bar\lambda) g \rangle_\H =  \langle  \varphi , 
	F_\lambda^*(\HIBC-\bar\lambda) g \rangle_{\dH} = -\langle  \varphi ,(A_m-I^*) g\rangle_{\dH}=0,
	\end{align}
	for all $\varphi\in \dH$, $g\in D(H_0)\subset D(\HIBC)$.
	
	To check (ii), notice that our previous analysis shows that for $\varphi\in D(T)$
	\begin{equation}
	F_\lambda \varphi = \Gamma_\lambda G_\lambda \varphi + f
	\end{equation}
	with $f\in D(\HIBC)\subset \ker(B-I^*)$. The claim thus follows from \autoref{lem:G Gamma}
	\begin{equation}
	(B-I^*)\Gamma_\lambda G_\lambda = (1-I^*G_\lambda)^{-1}- I^*G_\lambda (1-I^*G_\lambda)^{-1}=\Id_{D(T)}. \qedhere 
	\end{equation}	
\end{proof}

\begin{thm}\label{thm:IBC-QBT}
	The triple 
	\begin{equation*}
	\Big( \dH,	 (B-I^*), (A_m-I^*)\Big)
	\end{equation*}
 is a quasi boundary triple 
	for $H_m$. Furthermore, $\mathscr{G}:=\rg(A_m-I^*)\vert_{\ker(B-I^*)}$ is dense in $\dH$.
\end{thm}
\begin{proof}
	We have already shown the abstract Green identity (\autoref{lem:integration by parts-2}) and self-adjointness of 
	$H_m \vert_{\ker(B-I^*)}=\HIBC$ (\autoref{thm:HIBC self-adjoint}), so it only remains to prove that $\rg(B-I^*, A_m-I^*)$ 
	is dense in $\dH\times \dH$.
	To see this, first note that $\rg(B-I^*)=D(T)$ is dense. We can complete the argument by showing that $\mathscr{G}$ is 
	dense, since then the affine space $\{(A_m-I^*)f: f\in D(H_m), (B-I^*)f=\varphi\}$ is also dense for all $\varphi\in 
	\rg(B-I^*)$.
	To check this, it is sufficient to note that 
	\begin{equation}
	\left(\rg(A_m-I^*)R(\bar\lambda, \HIBC)\right)^\perp = \ker(F_\lambda),
	\end{equation}
	and that $F_\lambda$ is injective since it has the left-inverse $B-I^*$, by~\autoref{prop: Omega}.
\end{proof}

We can now use the theory of quasi boundary triples to obtain criteria for self-adjointness as well as a classification of interior-boundary conditions. We will formulate these in terms of linear realtions in $\dH$, i.e. linear subspaces of $\dH \oplus \dH$. This has the advantage of being able to deal with somewhat degenerate cases (e.g. where $\alpha=0$ or $\beta=0$) without distinction. We provide the relevant notions for calculating with relations in \hyperref[sect:relations]{Appendix \ref*{sect:relations}}.

We denote the Dirichlet-to-Neumann operator with respect to $A_m - I^*$ and $B - I^*$ by
\begin{equation*}
S_\lambda \coloneqq (A_m - I^*) F_\lambda, 
\end{equation*}
where $F_\lambda$ is defined in \autoref{def:F_lambda}. By~\autoref{prop: Omega}, $S_\lambda$ is well defined on $D(S_\lambda)=D(T)$ (since $D(A_m)\subset D(H_m)$). 
Following~\cite[Prop.2.4, Thm 2.8]{BeLa2007} we have (see~\autoref{prop:localRobin} for an application):

\begin{thm}\label{thm:genIBC-res} 
 Let $\mathfrak{R}$ be a linear relation in $\dH$ and define
 \begin{align*}
  H_\mathfrak{R}&=H_m\vert_{D(H_\mathfrak{R})} \\
  D(H_\mathfrak{R})&=\Big\{f\in D(H_m) : \big((B-I^*)f,(A_m-I^*)f\big)\in \mathfrak{R}\Big\}. 
 \end{align*}
If $\mathfrak{R}$ is symmetric, then $H_\mathfrak{R}$ is symmetric.

If moreover there exists a real $\lambda$ satisfying~\autoref{ass:prop 4.6} such that the relation $ \mathfrak{R}-S_\lambda$ is one-to-one and $\rg(F_\lambda^*)\subset \rg(\mathfrak{R}-S_\lambda)$, then $H_\mathfrak{R}$ is self-adjoint, $\lambda\in \rho(H_\mathfrak{R})$ and the resolvent is given by
\begin{equation*}
R(\lambda, H_\mathfrak{R})=(\Id + F_\lambda(\mathfrak{R}-S_\lambda)^{-1}(A_m-I^*) )R(\lambda, \HIBC).
\end{equation*}
\end{thm}
\begin{proof}
For $(f,g)\in \mathfrak{R}\subset \dH\times \dH$ and $(f^*, g^*)\in \mathfrak{R}^*$, then by definition of the adjoint relation~\eqref{eq:adjoint},
\begin{equation*}
 \langle f, g^* \rangle_{\dH} - \langle g, f^* \rangle_{\dH}=0.
\end{equation*}
Hence if $\mathfrak{R}\subset \mathfrak{R}^*$ the operator $H_\mathfrak{R}$ is symmetric by~\autoref{lem:integration by parts-2}.

By~\autoref{prop: Omega} we can write any $f\in D(H_m)$ uniquely as $f=f_0 + F_\lambda \varphi$ with $f_0\in D(\HIBC)=\ker(B-I^*)$, $\varphi\in D(T)$. As in~\autoref{lem:resLab}, solving $(\lambda-H_\mathfrak{R}) f=g$ then amounts to solving
\begin{align*}
 (\lambda - \HIBC) f_0 &\stackrel{!}{=} g\\
 \Big(\varphi, (A_m-I^*)f_0 + S_\lambda \varphi\Big) &\stackrel{!}{\in} \mathfrak{R}.
\end{align*}
The first equation and $\lambda\in \rho(\HIBC)$ imply
\begin{equation*}
 f_0=R(\lambda, \HIBC)g.
\end{equation*}
With $(A_m-I^*)f_0=F_\lambda^*g$, the inclusion is satisfied if and only if 
\begin{equation*}
 (\varphi, F_\lambda^*g) \in \mathfrak{R} - S_\lambda.
\end{equation*}
Since $\rg(F_\lambda^*)\subset \rg(\mathfrak{R} - S_\lambda)$, such a $\varphi$ exists and since $\mathfrak{R} - S_\lambda$ is one-to-one it is unique. Thus for every $g\in \dH$ we can uniquely solve for $f_0$ and $\varphi$, so $\lambda \in \rho(H_\mathfrak{R})$. Since 
\begin{equation*}
(\varphi, F_\lambda^*g)=(\varphi, (A_m-I^*)R(\lambda, \HIBC)g)\in  \mathfrak{R} - S_\lambda,
\end{equation*}
we have
\begin{equation*}
 ((A_m-I^*)R(\lambda, \HIBC)g, \varphi)\in ( \mathfrak{R} - S_\lambda)^{-1},
\end{equation*}
which, since $\varphi$ is unique, we write as $\varphi=(\mathfrak{R} - S_\lambda)^{-1}(A_m-I^*)R(\lambda, \HIBC)g$ and obtain the resolvent formula. Since $H_\mathfrak{R}$ is symmetric and its resolvent set contains the real number $\lambda$, $H_\mathfrak{R}$ is self-adjoint.
\end{proof}

In order to obtain a classification of the self-adjoint restrictions of $H_m$, which are extensions of $H_0$, we need the additional hypothesis that $D(H_0)$ is dense. To formulate this result, define the non-negative, bounded operator 
\begin{equation}
M:=(F_i^* F_i)^{1/2}
\end{equation}
We extend the operator $M^{-1}$ to $\dH$ as explained in~\hyperref[appendix]{Appendix \ref*{appendix}} and denote this extension by $M_-^{-1}$.
Applying the results of~\cite[Sect. 3]{BeMi14} to the boundary triple $(\dH,B - I^*,A_m - I^*)$ yields the following result (see \autoref{thm:appendix} for the relevant statement and a summary).

\begin{thm}\label{thm:BiMi14}
Assume that $D(H_0)$ is dense and that there exists $\lambda\in \R$ satisfying~\autoref{ass:prop 4.6}.  Let $\mathfrak{R}$ be a relation in $\dH$ and define $H_\mathfrak{R}$ as in \autoref{thm:genIBC-res}.

		Then $H_\mathfrak{R}$ is self-adjoint if and only if the relation
		\begin{equation}
		M^{-1}(\mathfrak{R}-S_\lambda)M_{-}^{-1}
		\end{equation}
		is self-adjoint and satisfies $D(\mathfrak{R})\subset M_{-}D(S)$.
\end{thm}

This is a complete classification, since for any self-adjoint operator $H$ with  $H_0\subset H \subset H_m$ there is a relation $\mathfrak{R}$ such 
that $H=H_\mathfrak{R}$, which is simply given by 
\begin{equation}
\mathfrak{R}=\{ \left((B-I^*)f, (A-I^*)f\right) \vert f\in D(H)\}.
\end{equation}
In the Moshinsky-Yafaev model (\autoref{exa:yaf}), $\dH=\C$ and one easily checks that $H_0$ is densely defined with deficiency indices $(1,1)$. The operators $\HIBCab$ (with $\alpha+\beta=1=\gamma+\delta$) are all self-adjoint extensions of $H_0$. In~\cite{yafaev1992}, these are discussed as extensions of the (not densely defined!) restriction of $L$ (and $H_0$) with boundary conditions $I^*f=0=Bf$. Our result clarifies their relation to the usual extension theory for symmetric operators.

\section{Applications}
\label{sect:applications}

\subsection{A toy quantum-field theory}

Here we illustrate our results in a simple model that displays much of the structure relevant for applications in quantum field theory, without posing too many technical problems for the verification of key assumptions.

The physical picture behind this example is that of a particle, whose position we denote by $x\in \R$, moving in a one-dimensional space while creating/annihilating ``particles''. The latter can be thought of as elementary excitations of the background medium through which the first particle moves. We denote their positions by $y_1, y_2, \dots$.
Such models play an important role in condensed matter physics. In the specific case we will consider, the excitations would not move on their own, although they will display effective dynamics through repeated creation/annihilation at different positions. This is analogous to the well known Fröhlich polaron model~\cite{GrWu16} to which the arguments of this section should apply with minor modifications.
Another, similar model with contact interactions in a three-dimensional space, which leads to some subtle regularity issues, was treated in~\cite{La21polaron}.

Take 
\begin{equation}
\H = \bigoplus_{n=0}^\infty L^2(\R)\otimes L^2_{\mathrm{sym}}(\R^n) = \bigoplus_{n=0}^\infty L^2\left(\R, 
L^2_{\mathrm{sym}}(\R^n)\right):= \bigoplus_{n=0}^\infty \H^{(n)}
\end{equation}
and $\dH=\H$ with $I=\Id$. Let $N$ be the operator given by $(N f)^{(n)}=nf^{(n)}$ (where $f^{(n)}$ is the projection of 
$f\in \H$ to $\H^{(n)}$), with the domain
\begin{equation}
D(N)=\{ f\in \H: \|n f^{(n)}\|_{\H^{(n)}} \in \ell^2(\N) \}.
\end{equation}
Clearly $N, D(N)$ is self-adjoint.
Let $x$ denote the first of the $n+1$ arguments of a function $f\in  \H^{(n)}$, then $-\Delta_x$ is a 
self-adjoint operator on the domain
\begin{equation}
D(-\Delta_x)= \bigoplus_{n=0}^\infty H^2\left(\R, 
L^2_{\mathrm{sym}}(\R^n)\right).
\end{equation}
We set $L=-\Delta_x +N$ with $D(L)=D(-\Delta_x)\cap D(N)$. 

We define $A:D(L)\cap \H^{(n)}\to \dH^{(n)}:=\H^{(n-1)}$ as a symmetrised evaluation operator ($A$ corresponds to the 
``annihilation operator'' $a(x)$):
\begin{equation}
\left(Af^{(n)}\right)(x,y_1, \dots, y_{n-1})=\frac{1}{\sqrt{n}} \sum_{j=1}^n f^{(n)}(x,y_1, \dots, y_{n})\vert_{y_j=x} 
= \sqrt n f^{(n)}(x,y_1, \dots, y_{n-1},x).
\end{equation}
One can check that $A$ maps $D(A):=D(L)$ to $\H$ and that $\ker A$ is dense. The operator $G_\lambda$ for $\lambda \in 
\C\setminus\R_+$ is then given by (denoting $Y=(y_1, \dots, y_{n+1})$ and $\hat Y_j$ as $Y$ without the entry $y_j$) 
\begin{align}
\left(G_\lambda f^{(n)}\right)(x,Y)
&=\frac{1}{\sqrt{n+1}} \sum_{j=1}^{n+1} (\lambda - L)^{-1} \delta(x-y_j)f^{(n)}(x, \hat Y_j)\\
&=\frac{1}{\sqrt{n+1}} \sum_{j=1}^{n+1} g_{n+1-\lambda}(x-y_j) f^{(n)}(y_j, \hat Y_j),\notag
\end{align}
with the function
\begin{equation}
g_\mu(x)= -(\mu-\Delta)^{-1}\delta = - \frac{e^{-\sqrt \mu |x|}}{2\sqrt{\mu}},
\end{equation}
where the square root is the branch with $\mathrm{Re}(\sqrt z)\geq 0$.
The operator $L_m$ is now defined on $D(L_m)=D(L)\oplus G_\lambda (\dH)$ by
\begin{equation}
L_m f = L_0^*f = L_0^*(f_0+G_\lambda \varphi) = Lf_0 + \lambda G_\lambda \varphi.
\end{equation}
The boundary operator $B$ is defined as the left inverse to $G_\lambda$ on $D(B)=D(L_m)$. 
In view of the fact that $H^2(\R)\subset C^1(\R)$ and $\lim_{r\to 0}(g_\mu'(r)-g_\mu'(-r))=1$, 
$B$ is given by the following local formula
\begin{equation}\label{eq:polaron B}
B f^{(n)}(x,Y)= \sqrt{n} \lim_{r\to 0} \left( \left(\partial_x f^{(n)}\right)(x+r,Y, x) -  \left(\partial_x 
f^{(n)}\right)(x-r,Y, x)\right),
\end{equation}
where the limit is taken in $\H^{(n-1)}$.

Since $g_\mu$ is continuous we can extend $A$ to $\rg G_\lambda$ canonically by using the same formula. This gives
\begin{align}
\left(T_\lambda f^{(n)}\right)(x,Y) &= \left(A_m G_\lambda f^{(n)}\right)(x,Y)  \\
&= - \frac{f^{(n)}(x,Y)}{2\sqrt{n+1-\lambda}} + \sum_{j=1}^n g_{n+1-\lambda}(x-y_j)f^{(n)}(y_j, \hat Y_j,x).\notag
\end{align}
Since $g_\mu$ is bounded, $T_\lambda:\H^{(n)} \to \H^{(n)}$ is a bounded operator. However, since the number of terms in 
the sum above is $n$, this does not give rise to a bounded operator on $\H$. 
We have the bound
\begin{equation}\label{eq:polaron T-est}
\|T_\lambda f^{(n)}\|_{\H^{(n)}} \leq \frac{n+1}{2|\sqrt{n+1-\lambda}|}  \|f^{(n)}\|_{\H^{(n)}},
\end{equation}
so on $D(T)=D(N^{1/2})$ we can define $T$ as an unbounded operator on $\H$.
This defines $A_m$ with domain $D(A_m)=D(L)\oplus G_\lambda (D(N^{1/2}))$.

\subsection{Self-adjointness of Robin-type operators}
\

The objects constructed above satisfy the hypothesis of our general setting, as explained in
\hyperref[construction]{Construction~\ref*{construction}}.
The operators with Robin type interior-boundary conditions $\HIBCab$ are thus well defined. The equation $\HIBCab f=g$ (with the choice $\delta=0$, $\gamma=\bar\beta^{-1}$, which is symmetric if $\alpha\bar\beta \in \R$) corresponds to the following hierarchy of boundary value problems
\begin{equation}
\left\{
\begin{aligned}
(-\Delta_x + n) f^{(n)}(x,Y) + \bar\beta^{-1}\sqrt{n+1} f^{(n+1)}(x, Y,x) & = g^{(n)}(x,Y)
\qquad &x\neq y_j\\
\alpha \sqrt{n} f^{(n)}(x, Y)+ \beta \sqrt{n} \left((\partial_xf^{(n)})_+ - (\partial_xf^{(n)})_-\right)(x, Y)& = f^{(n-1)}(x, \hat Y_n) \qquad &x=y_{n}
\end{aligned}\right.
\end{equation}
where the subscript $\pm$ indicates that the right/left sided limit $x\to y_n$ is taken, as in~\eqref{eq:polaron B}, and  $f^{(n)}$ is symmetric under permutation of $y_1, \dots, y_n$ which gives boundary conditions on the sets where $y_j=x$.

In order to establish self-adjointness of $\HIBCab$ we need some properties of $T_\lambda$. We remark that non-positivity of $T_\lambda$ is not generic in any way -- in~\autoref{exa:yaf} the operator is non-negative instead, while in the more involved cases studied in~\cite{thomas1984, LaSch19, La20} both the positive and negative parts are generally unbounded.

\begin{lem}\label{lem:polaron T}
	For any real $\lambda<0$, the operator $T_\lambda$ is essentially self-adjoint and non-positive. Moreover, $\rg (A)\subset \rg(z-T_\lambda)$ for all $z\in \C\setminus \R_-$. 
\end{lem}
\begin{proof}
	By the Sobolev embedding theorem $A$ has a natural extension to $D(L^{1/2})\cap \H^{(n)}\supset C(\R, L^2_{\mathrm{sym}}(\R^n))$, which we denote by $\tilde A$.
	For $\lambda< 0$, $L-\lambda$ is a positive operator, and we can then write
	\begin{equation}
	T_\lambda \vert_{\H^{(n)}}= A_m (A R(\lambda, L))^* 
	= - \left(\tilde A 
	(L-\lambda)^{-1/2}\right)\left(\tilde A(L-\lambda)^{-1/2}\right)^*,
	\end{equation}
	so $T_\lambda$ is symmetric and non-positive.
	
	Since $\H^{(n)}$ is $T_\lambda$-invariant,  $T_\lambda$ is an infinite direct sum of commuting bounded self-adjoint operators and thus essentially self-adjoint, since all vectors $f\in \H$ with only finitely many $f^{(n)} \neq 0$ are contained in $\rg(T_\lambda \pm i)$.
	
	Now let $f\in \rg(A)$, i.~e.~$f=A(L-\lambda)^{-1}g$ for some $g\in \H$, and $z\in \C\setminus\R_- \subset \rho(\overline{T}_\lambda)$.
	By the formula for $T_\lambda$, the operator 
	\begin{equation}
	R(z, \overline{T}_\lambda) \tilde A(L-\lambda)^{-1/2}
	\end{equation}
	is bounded, since multiplying by its adjoint from the right yields
	\begin{equation}
	- R(z, \overline{T}_\lambda)T_\lambda R(\bar z, \overline{T}_\lambda).
	\end{equation}
	Consequently  
	\begin{equation}
	N^{1/2}R(z, \overline{T}_\lambda) A(L-\lambda)^{-1}=R(z, \overline{T}_\lambda) A(L-\lambda)^{-1}(N+1)^{1/2}
	\end{equation}
	is also a bounded operator, and this shows that $(z-\overline{T}_\lambda)^{-1}\rg(A) \subset D(N^{1/2})=D(T)$ and thus $\rg(A)\subset \rg(z-T_\lambda)$.
\end{proof}

In particular this lemma shows that $T_{\bar\lambda}\subset T_\lambda^*$ for all $\lambda \in \C\setminus \R_+$, as assumed from~\hyperref[sect:Robin]{Section \ref*{sect:Robin}} on. 

In view of~\autoref{lem:resLab}, non-positivity of $T_\lambda$ together with $\rg(A)\subset \rg(z-T_\lambda)$ implies that $L_{\alpha, \beta}$ is self-adjoint with $\sigma(L_{\alpha, \beta})\subset [0, \infty)$ if ${\alpha}\bar\beta<0$ (and of course for $\alpha=0$, $\beta \neq 0$).
These operators correspond to repulsive contact interactions between the first particle and the remaining ones.

In order to make conclusions on $\HIBCab$, we need to verify the relevant hypothesis of~\autoref{lem:domains},~\autoref{lem:A rel HIBC-bounded}.
The fact that $1\in \rho(G_\lambda I^*)\cap \rho(I^*G_\lambda )$ is guaranteed by the following Lemma.

\begin{lem}\label{lem:polaron G}
	For $\lambda\in \C\setminus \R_+$ the operator $G_\lambda$ satisfies the bound 
	\begin{equation*}
	\| G_\lambda f^{(n)}\|_{\H^{(n+1)}} \leq 
	\frac{\sqrt{n+1}}{2|n+1-\lambda|^{1/2} \mathrm{Re}(\sqrt{n+1-\lambda})^{1/2}}\|f^{(n)}\|_{\H^{(n)}}
	\end{equation*}
	In particular for $\lambda<0$ we have $\|G_\lambda\|< \tfrac{1}{2}$.
\end{lem}
\begin{proof}
	Since $\|g_\mu \|_{L^2}=\frac{1}{2 \sqrt{|\mu|\mathrm{Re}(\sqrt{\mu})}}$ this follows from the triangle inequality.
\end{proof}

In view of~\autoref{cor:HIBC s-a alpha n= 0} this yields the following, with boundedness from below being a consequence of the non-negativity of $L_{\alpha, \beta}$ and the use of Kato-Rellich in the proof. 

\begin{prop}\label{prop.5.3}
	Let $\alpha\bar\beta<0$. For all $\delta, \gamma$ such that the symmetry condition of~\autoref{lem:symmetry} is satisfied, $\HIBCab$ is self-adjoint and bounded from below.
\end{prop}
\begin{proof}
	Let $\lambda<0$ and $\overline{T}_{\lambda}$ be the self-adjoint closure of $T_\lambda$. Then by~\autoref {prop:rechenregeln ab}~(v), $T_\lambda^{\alpha,\beta}$ is a restriction of the bounded operator $\gamma + R(-\beta,  \overline{T}_{\lambda})$, and this implies the relative bounds required in~\autoref{thm:HIBC self-adjoint}.
\end{proof}

To treat the case $\alpha=0$ we also need:

\begin{lem}\label{lem:polaron trafo}
	For all $\lambda \in \C\setminus\R_+$, the operators $G_\lambda I^*$, $\Gamma_\lambda=(1-G_\lambda I^*)^{-1}$ and $(1-I^* G_\lambda)^{-1}$ leave $D(N)$ as well as $D(N^{1/2})$ invariant. 
\end{lem}
\begin{proof}
	We give the proof only for $D(N)$, the proof for $D(N^{1/2})$ being essentially the same.
	For $G_\lambda I^*$ this is obvious, since 
	\begin{equation}
	NG_\lambda I^*=G_\lambda I^*(N+1).
	\end{equation}
	For $\Gamma_\lambda$, this follows by the same logic and~\autoref{lem:polaron G}, since $N (G_\lambda I^*)^4$ is bounded and
	\begin{equation*}
	\Gamma_\lambda= \sum_{k=0}^\infty(G_\lambda I^*)^k.
	\end{equation*}
	The argument for $(1-I^* G_\lambda)^{-1}$ is the same.
\end{proof}

\begin{rem}\label{rem:hierarchy}
 The argument of the lemma shows that, in the case of a hierarchy, it is not necessary that $G_\lambda I^*$  be small in norm for $\Gamma_\lambda$ to exist and be given by the power series. Rather, it is sufficient that
 \begin{equation*}
  G_\lambda I^*\vert_{\H^{(n)}}: \H^{(n)} \to \H^{(n+1)}
 \end{equation*}
has a norm that decreases with $n$, e.g., so that $\|G_\lambda I^*\vert_{\H^{(n)}}\|^n$ is summable. 
\end{rem}

\begin{prop}
	For $\alpha=0$, $\beta\neq 0$, and $\delta, \gamma$ such that the symmetry condition of~\autoref{lem:symmetry} is satisfied, $H_\mathrm{IBC}^{0, \beta}$ is self-adjoint and bounded from below.
\end{prop}
\begin{proof}
	We may pick any $0>\lambda \in \rho(L)$, and in view of~\autoref{cor:HIBC self-adjoint alpha = 0} it is sufficient to prove that $I T_\lambda I^*$ is infinitesimally $(\Id - G_{\lambda} I^*)^* (L-\lambda) (\Id - \Glambda I^*)$-bounded .
	
	By \eqref{eq:polaron T-est}, $I T_\lambda I^*$ is $N^{1/2}$-bounded. Because $N$ is $L$-bounded,  the invariance of $D(N)$ established in~\autoref{lem:polaron trafo} gives that $N$ is $(\Id - G_{\lambda} I^*)^* (L-\lambda) (\Id - \Glambda I^*)$-bounded. This implies the required infinitesimal bound by interpolation and proves the claim.
\end{proof}

We see that for $\alpha\bar\beta\leq 0$ all of the theorems of~\hyperref[sect:Robin]{Section \ref*{sect:Robin}} can be applied to this model. For the case $\alpha\bar \beta>0$, which corresponds to attractive interactions between the first particle and the remaining ones (in addition to the interaction induced by creation/annihilation of particles) this is not obvious and we do not know whether it holds.

\subsection{A pointwise Robin condition}\label{sect:localRobin}

We now discuss the applicability of our classification results of \hyperref[sect:ext]{Section~\ref*{sect:ext}}
and apply them to an example in which the coefficients $\alpha, \beta$ of the boundary condition are position-dependent. Such conditions were discussed in~\cite{tumulka2020}, though without proving self-adjointness.

By~\autoref{lem:polaron G}, $(1-I^*G_\lambda)^{-1}$ leaves $D(T)=D(N^{1/2})$ invariant. Hence the hypothesis of~\autoref{prop: Omega} and~\autoref{thm:IBC-QBT} are satisfied and 
$
\{\dH, (B-I^*), (A_m-I^*)\}
$
is a quasi boundary triple for 
\begin{equation}
H_m=L_m+ \Id + I(A_m - B).
\end{equation}
The self-adjoint restrictions of $H_m$ are described in~\autoref{thm:genIBC-res} and~\autoref{thm:BiMi14}.

A relevant class of boundary conditions are local versions of the boundary condition $\alpha Af + \beta Bf=I^*f$ where $\alpha, \beta$ are functions.

Let $\alpha, \beta\in L^\infty(\R)$ with $\alpha + \beta \equiv 1$ (see also~\autoref{rem:a+b=1} below).
The corresponding boundary condition reads 
\begin{equation}\label{eq:ibc-x}
\alpha (x) \left(A_m f^{(n+1)}\right)(x, Y) + \beta(x) \left(B f^{(n+1)}\right)(x, Y)=f^{(n)}(x,Y),
\end{equation}
for $n\in \N$ and $(x,Y)\in \R\times\R^n$.

\begin{rem}\label{rem:a+b=1}
	The condition $\alpha + \beta \equiv 1$ is less  restrictive than it might seem. Since our only requirement on $I$ is boundedness, any pair of functions with $(\alpha + \beta)^{-1}\in L^\infty$ can be accommodated by modifying $I$. 
	More precisely, set $\tilde I=(\bar\alpha + \bar\beta)^{-1}I$, then the condition~\eqref{eq:ibc-x} becomes,
	\begin{equation}
	\frac{\alpha (x)}{\alpha(x)+\beta(x)} A_m f^{(n+1)} + \frac{\beta(x)}{\alpha(x)+\beta(x)} B f^{(n+1)}=\frac{1}{\alpha(x)+\beta(x)}f^{(n)}= (\tilde I^*f)^{(n)},
	\end{equation}
	where the coefficients $\tilde \alpha:=\alpha(\alpha+\beta)^{-1}$, $\tilde \beta:=\beta(\alpha+\beta)^{-1}$ now satisfy $\tilde\alpha+\tilde \beta\equiv 1$.
\end{rem}

The relation corresponding to~\eqref{eq:ibc-x} is
\begin{equation}
\mathfrak{R}_{\alpha, \beta}=\{(\alpha f, -\beta f) \vert f\in \dH\}, 
\end{equation}
where $\alpha, \beta$ are the operators of multiplication by the respective function. To see this, note that $\left((B-I^*)f,(A_m-I^*)f\right)\in \mathfrak{R}_{\alpha, \beta}$ means that
\begin{equation}
\beta(B-I^*)f= -\alpha (A_m-I^*)f \Longleftrightarrow \alpha A_m f+ \beta B f= \underbrace{(\alpha + \beta)}_{=\Id} I^* f.
\end{equation}

Recall that 
self-adjointness of $H_{\mathfrak{R}_{\alpha, \beta}}$ is related to the operator
\begin{equation}
S_\lambda=(A_m-I^*)F_\lambda
\end{equation}
with $D(S_\lambda)=D(F_\lambda)=D(T)$, which has similar properties as $T_\lambda$.

\begin{lem}\label{lem:5.8}
	Let $\lambda_0:=\inf \sigma(\HIBC)$. Then for $\lambda<\lambda_0$, $S_\lambda$ is non-positive. 
\end{lem}
\begin{proof}
	For the form domains we have, similarly to~\autoref{lem:domains} (to prove this, take the closure of $D(\HIBC)$ in the appropriate norm), 
	\begin{equation}
	D\left(|\HIBC|^{1/2})\right)=\Gamma D(L^{1/2}).
	\end{equation}
	Since $g_\mu \in H^1(\R)$ in our example, we find $\Gamma D(L^{1/2})\subset D(L^{1/2})$.
	Let $\tilde A$ be the extension of $A$ to $D(L^{1/2})\cap \H^{(n)}$ for arbitary $n$ (note that $\tilde A (L-\lambda)^{-1/2}$ is not bounded, but $N^{1/2}$-bounded, due to the pre-factor $\sqrt{n}$ in the definition of $A$).
	For $\lambda<\lambda_0$, $-R(\lambda, \HIBC)=(\HIBC-\lambda)^{-1}>0$, so we have
	\begin{equation}\label{eq:polaron S}
	S_\lambda=- (\tilde A -I^*)(\HIBC-\lambda)^{-1/2}\left((\tilde A -I^*)(\HIBC-\lambda)^{-1/2}\right)^*\leq 0.
	\end{equation}
\end{proof}

\begin{prop}\label{prop:localRobin}
	Let $\alpha, \beta \in L^\infty$ with $\alpha+ \beta\equiv1$ and assume that there exists $\delta>0$ such that for all $x\in \R$
	\begin{equation*}
	\alpha(x)\bar\beta(x) \leq -\delta.
	\end{equation*}
	Then the operator $H_{\mathfrak{R}_{\alpha, \beta}}$ is self-adjoint.
\end{prop}
\begin{proof}
	We use the criterion of~\autoref{thm:genIBC-res}, i.e. we prove that $\mathfrak{R}_{\alpha, \beta} - S_\lambda$ is invertible on $\rg F_\lambda^*$ for $\lambda<\lambda_0$ (as above).

	Using the properties of $\alpha$ and $\beta$, we see that $|\alpha| \cdot |1-\alpha| \geq \delta$, so $\alpha^{-1}\in L^{\infty}$ and the relation can be rewritten as
	\begin{align}\label{eq:rel-ab S}
	\mathfrak{R}_{\alpha, \beta} - S_\lambda
	= & \{ (\alpha \varphi, -\beta \varphi - S_\lambda \alpha \varphi) | \alpha \varphi\in D(T) \} \\
	= & \{ (\varphi, -\beta \alpha^{-1} \varphi - S_\lambda \varphi) | \varphi\in D(T) \}.\notag
	\end{align}
	Let $\lambda<\lambda_0$ as above, and let $S_\lambda^F$ be the self-adjoint Friedrichs extension of $S_\lambda$. As $\beta \alpha^{-1}= \bar\beta \alpha (|\alpha|)^{-2} \leq - \delta \|\alpha\|_\infty^{-2}$, the operator 
	\begin{equation}
	\beta \alpha^{-1} + S_\lambda^F
	\end{equation}
	is strictly negative, self-adjoint and thus invertible. If $S_\lambda=S_\lambda^F$ we are finished since then the relation~\eqref{eq:rel-ab S} is invertible everywhere.

	If $S_\lambda \neq S_\lambda^F$ we can conlcude by showing that
	\begin{equation*}
	(-\beta \alpha^{-1} - S_\lambda^F)^{-1}\rg F^*_\lambda \subset D(T)=D(N^{1/2}),
	\end{equation*}
	since then  $S_\lambda^F$ can again be replaced by $S_\lambda$.
	This follows from the representation~\eqref{eq:polaron S} by the arguments of~\autoref{lem:polaron T}, since boundedness of 
	\begin{equation}
	(\beta \alpha^{-1} + S_\lambda^F)^{-1} (\tilde A - I^*) (\HIBC-\lambda)^{-1/2}
	\end{equation}
	together with the equality of $D(L^{1/2})$ and $D(|\HIBC|^{1/2})$ implies boundedness of
	\begin{equation}
	N^{1/2}(\beta \alpha^{-1} + S_\lambda^F)^{-1} F_\lambda^*. \qedhere 
	\end{equation}
	
\end{proof}


\appendix

\section{}

\subsection{Linear relations}
\label{sect:relations}

Here we briefly recall the relevant notions for linear relations in a Hilbert space $\H$. These generalise the corresponding notions for operators with the relation given by the graph. 
For a linear relation $\mathfrak{R}$ in $\H$ (i.e. a subspace of $\H\oplus\H$), the domain, range, kernel are defined by
\begin{align*}
D(\mathfrak{R})&=\{\varphi \in \H\vert \exists \eta \in \H: (\varphi, \eta)\in \mathfrak{R}\} \\
\rg(\mathfrak{R})&=\{ \varphi \in \H \vert \exists \psi \in \H: (\psi, \varphi)\in  \mathfrak{R}\} \\
\ker(\mathfrak{R})&=\{ \varphi \in \H \vert (\varphi, 0)\in  \mathfrak{R}\}.
\end{align*}
The following operations are defined on the set of linear relations in a Hilbert space $\H$
\begin{align*}
\mathfrak{R} + \mathfrak{S}&= \{(\varphi, \xi+\eta) \vert  (\varphi, \xi)\in \mathfrak{R},\, (\varphi, 
\eta)\in \mathfrak{S}\} \\
-\mathfrak{R}&= \{(\varphi, -\psi) \vert  (\varphi, \psi)\in \mathfrak{R}\} \\
\mathfrak{R}\mathfrak{S}&=\{(\varphi, \eta)\in \H\oplus\H \vert \exists \xi\in \H: (\varphi, \xi)\in \mathfrak{S} 
\text{ and} (\xi, \eta)\in \mathfrak{R}\}\\
\mathfrak{R}^{-1}&=\{(\varphi, \eta)\in \H\oplus\H \vert (\eta, \varphi)\in \mathfrak{R} \}.
\end{align*}
The adjoint relation is given by
\begin{equation}\label{eq:adjoint}
\mathfrak{R}^*:=\{(\varphi, \eta) \in \H\oplus\H\vert \forall (\psi, \xi)\in \mathfrak{R}: \langle \varphi, \xi 
\rangle_{\H} = \langle \psi, \eta \rangle_{\H}\}.
\end{equation}
A relation is symmetric if $\mathfrak{R} \subset \mathfrak{R}^*$ (as sets) and self-adjoint if 
$\mathfrak{R}=\mathfrak{R}^*$. Clearly an operator is self-adjoint if and only if its graph is a self-adjoint relation.

\subsection{Quasi-boundary triples}
\label{appendix}

In this section we briefly recall the definition of quasi boundary triples. Moreover, we translate results for quasi boundary triples in the language of our abstract framework. 

\begin{defn}\label{def:qbt}
	A triple $(\dH, \tilde{B}, \tilde{A})$ is called a \emph{quasi boundary triple} for an operator $\tilde{L} \colon D(\tilde{L}) \subset \H \to \H$ if
	$\dH$ is a Hilbert space and $\tilde{A}, \tilde{B} \colon D(\tilde{L}) \subset \H \to \dH$ are operators such that
	\begin{enumerate}[(i)]
		\item the second Green identity holds
		\begin{equation*}
			\langle \tilde{L} f, g \rangle_\H 
			- \langle f, \tilde{L} g \rangle_\H 
			= \langle \tilde{A} f, \tilde{B}g \rangle_{\dH}
			- \langle \tilde{B}f, \tilde{A}g \rangle_{\dH}
		\end{equation*}
		for all $f, g \in D(\tilde{L})$.
		\item The map $(\tilde{A}, \tilde{B}) : D(\tilde{L}) \to \dH\times\dH$ has dense range. 
		\item The restriction $L := \tilde{L}|_{\ker(\tilde{B})}$ is a self-adjoint operator on $\H$. 
	\end{enumerate} 
\end{defn}

This is equivalent to the definition given in~\cite[Def.2.1]{BeLa2007} in terms of relations, since by~\cite[Thm.2.3]{BeLa2007} the operator $\tilde L$ is always a restriction of (the relation) $L_0^*$ for the closed, but not necessarily densely defined, operator $L_0=L\vert_{\ker{A}}$. In other works, the additional assumption that $D(L_0)$ is dense is made, but we explicitly avoid this. 

By~\autoref{rem:boundary triple}, $(\dH, A_m, B)$ is a quasi boundary triple for $\tilde L=L\vert_{D(A_m)}$.
In \autoref{thm:IBC-QBT} we show that $\big( \dH,	 (B-I^*), (A_m-I^*)\big)$ is a quasi boundary triple for $H_m$.

\smallskip 

A quasi boundary triple is called an ordinary boundary triple if $\rg(\tilde{A},\tilde{B}) = \dH^2$ and a generalized boundary triple if $\rg(\tilde{B}) = \dH$. 

\smallskip

Note that a quasi boundary triple for ${\hat{L}}^\ast$ exists if and only if the defect indices $n_{\pm}(\hat{L}) := \dim(\ker(\hat{L}^\ast \mp i))$ of $\hat{L}^*$ coincide. Further, if the defect indices of $\hat{L}$ are finite the quasi boundary triple for $\hat{L}$ is an ordinary boundary triple.
Moreover, the operator $(\tilde{A},\tilde{B}) \colon D(\tilde{L}) \subset \H \to \dH \times \dH$ is closable and by \cite[Prop. 2.2]{BeLa2007} $\ker(\tilde{A},\tilde{B}) = D({\hat{L}})$ holds. By \cite[Thm. 2.3]{BeLa2007} it follows that $\tilde{L} = \hat{L}^\ast$ if and only if $\rg(\tilde{A},\tilde{B}) = \dH^2$. In this case the restriction $L := \hat{L}^\ast|_{\ker(\tilde{B})}$ is self-adjoint and the quasi boundary triple $(\dH,\tilde{B},\tilde{A})$ is an ordinary boundary triple.

\medskip 

For each $\lambda \in \rho(L)$ the definition of a quasi boundary triple yields the decomposition
\begin{equation}
	D(\tilde{L}) = D(L) \oplus \ker(\lambda - \tilde{L})
	\label{eq:decomposition:appendix}
\end{equation}
For $\tilde{L} = {L_m}|_{D(A_m)}$, $\tilde{B} = B$ the decompositions \eqref{eq:decomposition:appendix} and \eqref{eq:decomposition} coincide. While we used the Dirichlet operator $G_\lambda$ to obtain this decomposition, here the situation is contrariwise. Starting with the decomposition one obtains that the restriction $\tilde{B}|_{\ker(\lambda - \tilde{L})}$ is injective and $\rg(\tilde{B}|_{\ker(\lambda- \tilde{L})})  = \rg(\tilde{B})$. This yields to the following definition.

\begin{defn}
	Let $(\dH,B,A)$ a quasi boundary triple for $\tilde{L} \subset \hat{L}^\ast$.
	The \emph{$\gamma$-field corresponding to} $(\dH,\tilde{B},\tilde{A})$ is given by 
	\begin{equation*}
		\tilde{G} \colon \rho(L) \to \mathcal{L}(\dH,\H) \colon \lambda \mapsto (\tilde{B}|_{\ker(\lambda-\tilde{L})})^{-1} .
	\end{equation*}
	Moreover the \emph{Weyl function associated to} $(\dH,\tilde{B},\tilde{A})$ is given by
	\begin{equation*}
		\lambda \mapsto \tilde{T} (\lambda) := A \tilde{G}(\lambda) .
	\end{equation*}
\end{defn}

We point out that the $\gamma$-field $\tilde{G}(\lambda)$ at $\lambda \in \rho(L)$ equals to the abstract Dirichlet operator $G_\lambda$ by \autoref{prop:rechenregeln}~(i). Hence the Weyl function $\tilde{T}(\lambda)$ at $\lambda \in \rho(L)$ coincide with the abstract Dirichlet-to-Neumann operator $T_\lambda$.

\begin{rem}\label{rem:appendix}
	In \cite{BeLa2007} it is shown that for a boundary triple $(\dH,\tilde{B},\tilde{A})$
	\begin{equation}
	\tilde{T}(\lambda)^\ast \subset \tilde{T}(\bar{\lambda}) \label{eq:mlambdabar}
	\end{equation}
	for all $\lambda \in \rho(L)$ holds.
	Further, if $(\dH,\tilde{B},\tilde{A})$ is a generalized (in particular a ordinary) boundary triple, equality holds, i.e.
	\begin{equation}
	\tilde{T}(\lambda)^\ast = \tilde{T}(\bar{\lambda}) 
	\end{equation}
	for all $\lambda \in \rho(L)$. 
\end{rem}

Note that our approach is contrariwise. We start with \autoref{ass:general} and show that $(\dH,B,A_m)$ is a quasi boundary triple for ${L_m}|_{D(A_m)}$. Whereas in \cite{BeLa2007} they start with a quasi boundary triple and see that \eqref{eq:mlambdabar} holds. 

\medskip 

Consider a Hilbert space $\mathcal{H}$ and a closed, densely defined operator $U \colon D(U) \subset \mathcal{H} \to \mathcal{H}$ with $\rho(U) \not = \emptyset$.
Replacing $U$ by $U-\lambda$ for $\lambda \in \rho(U)$ we assume without loss of generality that $0 \in \rho(U)$. 
Now we define the norm $\| \phi \|_{-1} := \| U^{-1} \phi \|_{\mathcal{H}}$
and by $\mathcal{H}_{-1} := (\mathcal{H}, \| \cdot \|_{-1})^{\sim}$ the completion of $\mathcal{H}$ with respect to the $\| \cdot \|_{-1}$-norm. 
Now $\mathcal{H}_{-1}$ equipped with the $\| \cdot \|_{-1}$-norm is a Banach space and if $U$ is symmetric a Hilbert space. 
Further, we define $\mathcal{H}_1$ by $D(U)$ equipped with the graph norm. 
Denote by $U_{-1}^{-1}$ the unique extension of $U^{-1}$ from $\mathcal{H}_{-1}$ to $\mathcal{H}$. 
We point out that $U^{-1}$ is an isometry from $\mathcal{H}_{1} \to \mathcal{H}$ and $U_{-1}^{-1}$ is an isometry from $\H \to \H_{-1}$, i.e.,
\begin{equation*}
\mathcal{H}_{-1} \stackrel{U_{-1}^{-1}}{\longrightarrow} \mathcal{H} \stackrel{U^{-1}_{\phantom{-1}}}{\longrightarrow} \mathcal{H}_{1}.
\end{equation*}

If $U$ is generator of a strongly continuous semigroup the space $\H_{-1}$ is the extrapolation space of order $-1$ associated to $U$ (see \cite[Def.~II.5.4]{EN:00}). We refer to \cite[Sect.~II.5]{EN:00} for more details about extrapolation spaces.

\smallskip  

Now consider the quasi boundary triple $(\dH, \tilde A, \tilde B)$ for $\tilde L$. Let $G_\lambda$ be the associated family of abstract Dirichlet operators (the $\gamma$-field). 
Assume that $\mathscr{G}:=\rg \tilde A\vert_{D(L)}$ is dense. Then 
\begin{equation}
M:=(G_i^* G_i)^{1/2}
\end{equation}
defines a positive, injective operator on $\dH$. By~\cite[Prop. 2.9]{BeMi14} we have that $\mathscr{G}=\rg G_i^* = \rg M$. Note that $M^{-1} \colon \mathscr{G} \subset \dH \to \dH$ is an invertible, densely defined, closed operator. We set $U \coloneqq M^{-1}$ and $\mathscr{G}_+ \coloneqq \mathscr{G}_{1}$, $\mathscr{G}_- \coloneqq \mathscr{G}_{-1}$. Then the spaces are Hilbert spaces. Denote the unique extension $M_{-} \coloneqq (M^{-1})_{-1}^{-1} \colon \mathscr{G}_- \to \dH$ we
obtain the isometries
\begin{equation*}
 \mathscr{G}_- \stackrel{M_{-}}{\longrightarrow} \dH \stackrel{M_{\phantom{ }}}{\longrightarrow} \mathscr{G}_+.
\end{equation*}
Denote by $\bar{L}$ the closure of $\tilde{L}$ and by $\bar{B} \colon D(\bar{L}) \to \mathscr{G}_-$ the unique extension of $\tilde{B}$. 
Moreover \eqref{eq:decomposition:appendix} yields
the decomposition
\begin{equation}
	D(\bar{L}) = D(L) \oplus \ker(\lambda - \bar{L})
\end{equation}
for $\lambda \in \rho(L)$ and we denote the projection onto $D(L)$ by $\pi$. 
Define $\hat{B} \coloneqq M_{-} \circ \bar{B} \colon D(\bar{L}) \to \dH$ and $\hat{A} \coloneqq M^{-1} \circ \tilde{A} \circ \pi \colon D(\bar{L}) \to \dH$.
By \cite[Thm.~2.12]{BeMi14} $(\dH,\hat{A},\hat{B})$ is an ordinary boundary triple for $\bar{L}$. 
This allows to extend the classification of ordinary boundary triples (see \cite[Thm.~3.1]{BeMi14}) to quasi boundary triples (see \cite[Thm.~3.4]{BeMi14}).
For more details we refer to \cite[Sect.~3]{BeMi14}. See also \cite{derkach1995}. 

\begin{thm}[{\cite[Cor.~3.5]{BeMi14}}]\label{thm:appendix}
	Let $(\dH,\tilde{B},\tilde{A})$ a quasi boundary triple for $\tilde{L}$. Assume that there exists a $\lambda \in \R \cap \rho(\tilde{L})$. Let
	$\mathfrak{R}$ be a relation in $\dH$. Then
	$L|_{\mathfrak{R}}$ given by
	\begin{align*}
	L_\mathfrak{R}&=H\vert_{D(L_\mathfrak{R})} \\
	D(L_\mathfrak{R})&=\Big\{f\in D(\tilde{L}) : \big(\tilde{B}f,\tilde{A}f\big)\in \mathfrak{R}\Big\}. 
	\end{align*}
	is self-adjoint if and only if the relation
	\begin{equation*}
		M^{-1} (\mathfrak{R}-T(\lambda)) M_{-}^{-1}
	\end{equation*}
	is self-adjoint and satisfies $D(\mathfrak{R}) \subset M_- D(T)$. 
\end{thm}

\bigskip

\subsection*{Acknowledgements}

The authors thank the Department of Mathematics  at the Ludwig-Maximilians University of Munich, where part of this research was conducted, for its hospitality.

%

\vspace*{5em}




\end{document}